\documentclass[11pt,a4paper,reqno,dvipsnames]{amsart}
\usepackage{geometry}
\usepackage{amsfonts}
\usepackage{amsmath}
\usepackage{amssymb,stmaryrd}
\usepackage{amsthm}
\usepackage{amsxtra}
\usepackage{bbm}
\usepackage{braket}
\usepackage{comment}
\usepackage{extarrows}
\usepackage{graphicx}
\usepackage{latexsym}
\usepackage{mathrsfs}
\usepackage{yhmath}
\usepackage{nicematrix}
\usepackage{mathtools}

\usepackage{float}
\usepackage{verbatim}
\usepackage{listings}
\usepackage[normalem]{ulem}

\usepackage[pdfborder={0 0 0}]{hyperref} 
\hypersetup{colorlinks, citecolor=blue, linkcolor=blue, urlcolor=blue}
\hypersetup{pagebackref}
\usepackage{bm}

\usepackage{tikz}

\usepackage[initials,lite]{amsrefs} 
\AtBeginDocument{%
    \def\MR#1{}
}

\usepackage{cleveref}

\makeatother

\theoremstyle{plain}

\newtheorem{Theorem}{Theorem}[section]
\newtheorem{Lemma}[Theorem]{Lemma}
\newtheorem{Corollary}[Theorem]{Corollary}
\newtheorem{Proposition}[Theorem]{Proposition}

\theoremstyle{definition}

\newtheorem{Assumptions and Discussion}[Theorem]{Assumptions and Discussion}

\newtheorem{Example}[Theorem]{Example}
\newtheorem{Definition}[Theorem]{Definition}

\newtheorem{Remark}[Theorem]{Remark}

\newtheorem{Notation}[Theorem]{Notation}

\newtheorem{Setting}[Theorem]{Setting}
\theoremstyle{remark}

\newtheorem*{acknowledgment*}{Acknowledgment}

\usepackage[shortlabels]{enumitem}
\SetEnumerateShortLabel{a}{\textup{(\alph*)}}
\SetEnumerateShortLabel{A}{\textup{(\Alph*)}}
\SetEnumerateShortLabel{1}{\textup{(\arabic*)}}
\SetEnumerateShortLabel{i}{\textup{(\roman*)}}
\SetEnumerateShortLabel{I}{\textup{(\Roman*)}}

\def\Char{\operatorname{char}}

\def\deg{\operatorname{deg}}

\def\dim{\operatorname{dim}}

\def\ini{\operatorname{in}}

\def\ker{\operatorname{ker}}
\def\KK{{\mathbb K}}

\def\Lift{\operatorname{Lift}}

\def\NN{{\mathbb N}}

\def\part{\operatorname{part}}

\def\Quad{\operatorname{Quad}}

\def\RR{{\mathbb R}}

\def\sgn{\operatorname{sgn}}

\def\sL{\mathscr{L}}

\def\sort{\operatorname{sort}}

\def\Sym{\operatorname{Sym}}

\def\trdeg{\operatorname{tr.deg}}

\def\ZZ{{\mathbb Z}}

\newcommand\bda{{\bm a}}

\newcommand\bdb{{\bm b}}

\newcommand\bdc{{\bm c}}
\newcommand\bdD{{\bm D}}
\newcommand\bdd{{\bm d}}

\newcommand\bde{{\bm e}}

\newcommand\bdg{{\bm g}}

\newcommand\bdk{{\bm k}}

\newcommand\bdl{{\bm \ell}}

\newcommand\bdu{\bm{u}}

\newcommand\bdx{{\bm x}}

\newcommand\bfc{\mathbf{c}}
\newcommand\bfC{\mathbf{C}}

\newcommand\bfL{\mathbf{L}}

\newcommand\bfT{\mathbf{T}}

\newcommand\bfw{\mathbf{w}}
\newcommand\bfx{\mathbf{x}}
\newcommand\bfX{\mathbf{X}}

\newcommand\bfZ{\mathbf{Z}}

\newcommand\calA{\mathcal{A}}

\newcommand\calC{\mathcal{C}}
\newcommand\calD{\mathcal{D}}

\newcommand\calF{\mathcal{F}}
\newcommand\calG{\mathcal{G}}

\newcommand\calI{\mathcal{I}}
\newcommand\calJ{\mathcal{J}}
\newcommand\calK{\mathcal{K}}
\newcommand\calL{\mathcal{L}}

\newcommand\calR{\mathcal{R}}

\newcommand\calU{\mathcal{U}}

\def\frakS{\mathfrak{S}}

\definecolor{MyGreen}{RGB}{34,136,51}

\begin{document}

\title{Blowup algebras of determinantal modules}

\author[Kuei-Nuan Lin, Yi-Huang Shen]{Kuei-Nuan Lin and Yi-Huang Shen}

\thanks{2020 {\em Mathematics Subject Classification}.
    Primary 13C40, 
    13A30, 
    13P10, 
    Secondary 14N07, 
    14M12 
}

\thanks{Keyword: Multi-Rees algebra, determinantal, Cohen--Macaulay, normality, singularity}

\address{Department of Mathematics, The Penn State University, McKeesport, PA, 15132, USA}
\email{linkn@psu.edu}

\address{
School of Mathematical Sciences, University of Science and Technology of China, Hefei, Anhui, 230026, P.R.~China}
\email{yhshen@ustc.edu.cn}

\begin{abstract}
    We study the blowup algebras of the modules that are direct sums of ideals generated by either maximal minors of a ladder matrix or unit interval determinantal ideals. Specifically, we determine Gr\"{o}bner bases for the presentation ideals of multi-Rees algebras and their special fiber rings. Our analysis reveals that the multi-blowup algebras are Koszul Cohen--Macaulay normal domains, possess rational singularities in characteristic zero, and are F-rational in positive characteristic.
\end{abstract}

\maketitle

\section{Introduction}

Rees algebras, which have been extensively discussed in \cite{vasconcelos1994arithmetic}, occupy a central position in commutative algebra and algebraic geometry. Their applications extend to fields such as elimination theory, geometric modeling, chemical reaction networks, and algebraic statistics \cite{CWL, Cox, CLS, Likelihood}. Equally significant are the determinantal varieties in these fields, as highlighted in \cite{BCRV}.

The classical concept of the Rees algebra of an ideal can be extended to the Rees algebra of a module. When $M$ is a finitely generated module with a rank over a Noetherian ring $R$, the \emph{Rees algebra} $\mathcal{R}(M)$ is defined as the symmetric algebra of $M$ modulo its $R$-torsion submodule (see \cite{EHU}). In the case where $M$ is either the module of differentials or the module of top differential forms of the homogeneous coordinate ring of a projective $\KK$-variety $Y$, the \emph{special fiber ring} $\KK \otimes \mathcal{R}(M)$ is the homogeneous coordinate ring of the tangential variety or the Gauss image of $Y$, respectively (see \cite{SUVModule}).

In this note, we look at the special case where $M = \bigoplus_{i=1}^r I$ is the direct sum of a homogeneous ideal $I$ in the polynomial ring $R = \KK[x_1, \ldots, x_n]$ over a field $\KK$. In this context, the Rees algebra
is also known as the \emph{multi-Rees algebra} of the ideal $I$. Similar to the classical case, it can be defined as the multi-graded $R$-subalgebra 
\[
    \mathcal{R}(M) \coloneqq R[I t_1, \ldots,I t_r ]
    \subseteq R[t_1, \ldots, t_r],
\]
with auxiliary variables $t_1, \ldots, t_r$ introduced. This algebra is the homogeneous coordinate ring of the blowup along the subschemes defined by these ideals. The associated \emph{special fiber ring}, denoted by $\mathcal{F}(M)$, corresponds to the image resulting from the blowup map.
These two blowup algebras can be represented as the quotients of polynomial rings:
\[
    \mathcal{R} (M) \cong S[T_1,\ldots, T_\mu]/\mathcal{J} \qquad \text{and} \qquad \mathcal{F} (M) \cong \mathbb{K}[T_1,\ldots, T_\mu]/\mathcal{K},
\]
where $\mu$ is the total number of minimal generators of $M$.
One significant challenge in studying the multi-Rees algebras is determining the implicit equations of the presentation ideals $\mathcal{J}$ and $\mathcal{K}$. Cox, Lin, and Sosa showed that a comprehensive understanding of $\mathcal{J}$ and $\mathcal{K}$ can illuminate the equilibrium solutions of chemical reaction networks \cite{CLS}. 
The problem of identifying the generating equations of $\mathcal{J}$ and $\mathcal{K}$ poses a substantial challenge in the fields of elimination theory and geometric modeling. For further details, please see, for example, \cite{Cox}.

It is readily seen that $\mathcal{F}(M)$ is isomorphic to the Segre product $\mathcal{F}(I) \#_{\KK} \KK[t_1,\ldots, t_r]$. Therefore, to establish the Koszul, normal, and similar properties for $\calF(M)$, it suffices to do so for $\calF(I)$; see for instance \cite{GWGraded,MR789425}. One may also argue that $\calR(M)$ is isomorphic to the Segre product $\calR(I) \#_{R} R[t_1,\ldots, t_r]$. Notice that this Segre product is taken over $R$, where the graded rings have $R$ in the $0$-th degree. Since this is not about the Segre product of two graded rings defined over a field, its treatment does not fall into the framework of the classical papers \cite{GWGraded,MR789425}.

The study of determinantal ideals is another complex undertaking. Although the multi-Rees algebras for direct sums of maximal minors of a generic matrix have been scrutinized by Burns and Conca \cite{BCresPowers}, the problem of implicit equations for the multi-Rees algebras of direct sums of general determinantal ideals remains unresolved.

In this study, we will investigate the multi-blowup algebras of two classes of determinantal ideals. Specifically, we will extend the findings of Celikbas et al.~\cite[Theorem 1.1]{CDFGLPS} and Almousa et al.~\cite[Theorem 4.5]{ALL1} to the realm of multi-blowup algebras.
The conclusions we present are grounded in the understanding of the blowup algebras 
associated with maximal minors of generic matrices.

Consider a generic matrix $\mathbf{X}$, where the entries are distinct variables. Let $I$ denote the ideal of all its maximal minors. It is well-known that the special fiber ring $\mathcal{F}(I)$ of $I$ forms the coordinate ring of a Grassmannian variety. Moreover, both $\mathcal{R}(I)$ and $\mathcal{F}(I)$ are Cohen--Macaulay algebras.
The normality of $\mathcal{F}(I)$ then follows from its Cohen--Macaulay property, given the smoothness of the Grassmannian variety.
Moreover, the powers and symbolic powers of $I$ coincide, implying that $\mathcal{R}(I)$ is normal. Furthermore, the deformation theorem shows that both \( R(I) \) and \( F(I) \) have quadratic Gröbner bases.
Consequently, both blowup algebras are Koszul, see details in \cite[Chapter 6]{BCRV}. Then by \cite{Blum}, all powers of the ideal $I$ have a linear resolution.

Let $L=I_n(\bfX_{\bfL})$ be the ideal generated by the maximal minors of a ladder matrix $\bfX_{\bfL}$ (see Definition \ref{def:Ladder}). 
When the initial ideal of $L$ can be associated with a Ferrers diagram, the blowup algebras have been studied by \cite[Proposition 4.3]{CDFGLPS}. Here,
we will investigate both $\mathcal{R} (\oplus_{i=1}^r L)$ and $\mathcal{F} (\oplus_{i=1}^r L)$. Additionally, let $U$ be a unit interval determinantal ideal (refer to Definition \ref{def:unit}). Our work will take care of $\mathcal{F} (\oplus_{i=1}^r U)$ and  $\mathcal{R} (\oplus_{i=1}^r U)$.

It is important to point out that although the ideal $L$ is associated with the ladder matrix $\bfX_{\bfL}$, it is not the ladder determinantal ideals of minors in the literature.
Recall that this notion was introduced and studied by Abhyankar~\cite{MR926272}, and subsequently by other authors. For example, Conca and De Negri showed that the Rees algebra of a ladder determinantal ideal coincides with the symmetric algebra~\cite[Theorem~5.2]{CD}.
In contrast, even when $r=1$, the Rees algebra of $L$ may not be equal to its symmetric algebra.
In a recent contribution, Ramos and Simis identified the presentation ideal of the special fiber of cofactors for a given square ladder \cite[Corollary 3.11]{RSMatrix}. 
However, their work pertains to a completely different class of ideals and does not encompass considerations of multi-Rees algebras. 

The following is an outline of the road map. 
We first establish basic terminology and notation in \Cref{sec:prelim}. In \Cref{fibertoRees}, we provide a $\KK$-algebra isomorphism between the multigraded Segre product of $\calR(J)$ and $\calR(\bigoplus_{i=1}^r J)$ for a monomial ideal $J$. We then use the toric fiber product to extract a Gr\"{o}bner basis of the Segre product in \Cref{GBMultiRessInitial_FiberProduct}.
After that,
we start with a natural distributive lattice structure from the index set of the maximal minors, whose Hibi ring is isomorphic to the toric ring related to the initial ideal of the maximal minors of the generic matrix. We prove that a non-empty sublattice induces an isomorphism of its Hibi ring to the associated toric subring in \Cref{CanUseRestriction}. The structure of the ladder matrix or the unit interval characteristic ensures that we have the sublattice property (\Cref{lem:LexClosed} and \Cref{minMaxExist}). Therefore, we obtain quadratic Gr\"{o}bner bases for the fiber cones of the direct sum of the initial ideals of $L$ and $U$ (see Theorems \ref{FiberLadder} and \ref{FiberUnit}). 
In contrast to the constructive methodologies proposed by Almousa et al.~\cite{ALL1} and Celikbas et al.~\cite{CDFGLPS}, 
this approach is inspired by the work of De Negri \cite{DeNegri} for the fiber ring of a lexsegment ideal.
Our proof strategy is considerably more concise and efficient than those in \cite{ALL1} and \cite{CDFGLPS} with regard to 
the special fiber of the ideal cases, while also recovering their results. 
As a matter of fact, in \cite{ALL1}, the authors use the so-called sorting order and clique-sorted structure to represent the comparable pairs in the Hibi ideal, which demands cumbersome notations and more involved index tracking. 
Subsequently, we leverage the $\ell$-exchange property to derive the Gr\"{o}bner basis of the Rees algebra of the initial ideals of maximal minors of ladder matrices. We then use the Segre product to obtain the Gr\"{o}bner basis of the multi-Rees algebra of the direct sum of these initial ideals (see \Cref{fibertoRees}, \Cref{GBMultiRessInitial_FiberProduct}, and \Cref{thm:InitialMultiRees}). 

In the final sections of our work, we employ the \textsc{Sagbi} (Khovanskii) technique to derive the defining equations of the multi-Rees algebras and their fibers of the original determinantal modules.
Consequently, we can demonstrate that 
presentation ideals of these multi-blowup algebras have
quadratic Gr\"{o}bner bases (see Theorems \ref{FiberLadder} and \ref{thm:sagbiLiftsUnit}). 
Moreover, these algebras are Koszul normal Cohen--Macaulay domains, and we provide a classification of their singularities (see Corollaries \ref{cor:propertiesReesLadder} and \ref{thm:sagbiLiftsUnit}).

The main focus of this work is to provide an explicit description of the presentation ideals of the multi-Rees algebras and their corresponding fibers for the original determinantal modules. A byproduct of our approach is that many important properties of these algebras emerge naturally using the distributive lattice structure, allowing Hibi ring theory to be applied. 
In contrast, the existing literature about finding properties of Rees algebra often derives these properties using powers of ideals or modules; see for example, \cite{BCresPowers} and \cite{LinReesModule}.

The explicit descriptions of presentation ideals have practical significance in real-world applications. Our approach, based on Gr\"{o}bner bases and \textsc{Sagbi} bases, enables the 
application of Gr\"{o}bner deformation and \textsc{Sagbi} deformation to compute numerical invariants, such as regularity and multiplicity, of these algebras via the initial ideals of their defining ideals; see, for example, \cite{wicaII} and \cite{LSCMNormal}.

Our main results are summarized below.

\begin{Theorem}
    Suppose that $\bfX_{\bfL}$ is a  ladder matrix obtained from a generic matrix $\bfX$. Let $L=I_n(\bfX_{\bfL})$ be the 
    determinantal ideal generated by the maximal minors of $\bfX_{\bfL}$. Meanwhile, assume that $U$ is a unit interval determinantal ideal arising from some of the maximal minors of $\bfX$. Let $M_L=\oplus_{i=1}^r L$ and $M_U=\oplus_{i=1}^r U$ be the modules associated to those ideals. 
    \begin{enumerate}[a]
        \item The presentation ideals of $\mathcal{F}(M_L)$ and $\mathcal{F}(M_U)$ are generated by the inherited Pl\"ucker relations, respectively. These relations form a quadratic Gr\"obner basis of the ideal they generate.
        \item The multi-Rees algebra $\mathcal{R}(M_L)$ is of fiber type, meaning that its presentation ideal is generated by the relations of the symmetric algebra of $M_L$ and by the Pl\"ucker relations on the maximal minors of $\bfX_{\bfL}$. Moreover, these equations form a Gr\"obner basis of the ideal they generate.
        \item The blowup algebras $\mathcal{R}(M_L)$, $\mathcal{F}(M_L)$, $\mathcal{R}(M_U)$, and $\mathcal{F}(M_U)$ are Cohen--Macaulay normal domains. Furthermore, they have rational singularities in characteristic zero and are $F$-rational in positive characteristic. 
        \item The blowup algebras $\mathcal{R}(M_L)$, $\mathcal{F}(M_L)$, $\mathcal{R}(M_U)$, and $\mathcal{F}(M_U)$ are Koszul.
            In particular, the powers of $L$ have a linear resolution.
        \item The natural algebra generators of $\mathcal{R}(M_L)$, $\mathcal{F}(M_L)$, $\mathcal{R}(M_U)$, and $\mathcal{F}(M_U)$ are \textsc{Sagbi} bases of the algebras they generate.
    \end{enumerate}
\end{Theorem}

\section{Preliminaries on the multi-Rees algebra and its special fiber}
\label{sec:prelim}

We start by going over the basic facts and notions that we will employ in this work. 
In particular, we will take the following convention for the notations.

\begin{Notation}
    \begin{enumerate}
        \item To avoid confusion, the set of positive integers is denoted by $\ZZ_+$, and the set of non-negative integers is denoted by $\NN$.

        \item If $p$ and $q$ are two positive integers such that $p\le q$, we denote the set $\{p,p+1,\dots,q-1,q\}$ by $[p,q]$. If additionally $p=1$, we will write it as $[q]$ for brevity.

    \end{enumerate}
\end{Notation}

\subsection{Blowup algebras}
\label{subsection:blowup_algebras}
In this subsection, we explain in more detail the presentation ideals of the blowup algebras, tailored to our determinantal setting.

Let $m,n$, and $r$ be three positive integers. Suppose that $\bfX=(x_{i,j})$ is a generic $n\times m$ matrix with the entries $x_{1,1},\dots,x_{n,m}$ being indeterminates over a fixed field $\KK$.
Let $R = \KK[\bfX]$ be the polynomial ring over $\KK$ in these variables.
Suppose that $I$ is an equigenerated ideals in $R$, and $\{f_{1}, \dots,f_{\mu}\}$ is a minimal homogeneous generating set of $I$. In the following, we will consider the blowup algebras associated with $M=\bigoplus_{j=1}^r I$.

To describe the presentations of the \emph{Rees algebra} 
\[
    \mathcal{R}(M)\coloneqq R[I t_1, \ldots,I t_r ]
    =\bigoplus_{a_i \in \mathbb{N}} I^{a_1}\cdots I^{a_r}t_1^{a_1}\cdots t_r^{a_r} 
    \subseteq R[t_1,\ldots,t_r],
\] 
and the \emph{special fiber ring} 
\[
    \mathcal{F}(M) \coloneqq \calR(M) \otimes_{R} \KK \cong 
    \KK[f_{i}t_j\mid i\in [\mu],j\in [r]]
\] 
of this module, 
we introduce two polynomial rings over the field $\KK$: 
\[
    R[\bfT]\coloneqq R[T_{i,j}\mid i\in [\mu],j\in [r]] \qquad\text{and}\qquad  \KK[\bfT] \coloneqq \KK[T_{i,j}\mid i\in [\mu],j\in [r]].
\]
From them, we have
two surjective ring homomorphisms:
\begin{align*}
    \phi: R[\bfT] \to \mathcal{R}(M) \quad \text{and} \quad \psi: \KK[\bfT] \to \mathcal{F}(M).
\end{align*}
Here, for each $l\in [n]$, $ k\in [m]$, $i\in [\mu_j]$ and $j\in [r]$, we require that $
\phi(x_{l,k})= x_{l,k}$, and $
\phi(T_{i,j})=\psi(T_{i,j})=f_{i,j} t_j$. 

\begin{Definition}
    \begin{enumerate}[a]
        \item The kernel ideals $\mathcal{J}\coloneqq \ker(\phi)$ and $\mathcal{K}\coloneqq \ker (\psi)$ are called the \emph{presentation ideals} of $\mathcal{R}(M)$ and $\mathcal{F}(M)$, respectively. Alternatively, they are sometimes referred to as 
            the \emph{Rees ideal} and the \emph{special fiber ideal}, respectively. 
            Meanwhile, the \emph{defining equations} of an algebra form a system of generators for its presentation ideal.
        \item If $\mathcal{K}S$ together with the relations of the symmetric algebra 
            $\Sym(M)$
            generate $\mathcal{J}$, then $\calR(M)$ is \emph{of fiber type}.  
    \end{enumerate}
\end{Definition}

From now on, in this paper, all unspecified tensor products $\otimes$ are taken over the common ground field $\mathbb{K}$.

Let $A$ and $B$ be two graded rings over $\KK$.
To recover the presentation ideal of the Segre product of $A$ and $B$
using the presentation ideals of $A$ and $B$ is generally not as straightforward as one might normally expect, see for example \cite{KRSegre, Yang_Baxter_algebras}.
In multigraded case, Kahle and Rauh show in \cite{KRSegre} that 
the Segre product is presented by the toric fiber product of two ideal, under some assumptions. The case that we are interested in this paper can take advantage of this observation. As a result, we establish notations for Rees algebra of initial ideals and recall some of the work related to the toric fiber product and Segre product here.

In the context under consideration, the set $\{f_{1}, \dots, f_{\mu}\}$ is also assumed to be a Gr\"obner basis of degree $d$ for the ideal $I$ with respect to a term order $\tau$. Consequently, we may define two additional ring homomorphisms

\[
    \phi^\ast: R[T_1,\ldots, T_{\mu}] \longrightarrow \mathcal{R}(\ini_{\tau} (I)) \quad \text{and} \quad \psi^\ast: \mathbb{K}[T_1,\ldots, T_{\mu}] \longrightarrow \mathcal{F}(\ini_{\tau} (I))
\]
for the Rees algebra $\mathcal{R}(\ini_{\tau} (I))$ and the special fiber ring $\mathcal{F}(\ini_{\tau} (I))$. Here, for all $l\in [n]$ and $k\in [m]$, let $\phi^\ast(x_{l,k}) = x_{l,k}$ and $\phi^\ast(T_{i}) =\psi^\ast(T_{i})=\ini_{\tau}(f_{i})t$.
Then $\calJ ^\ast=\ker(\phi^\ast)$ and $\calK^\ast=\ker(\psi^\ast)$ are presentation ideals of $\mathcal{R}(\ini_{\tau} (I))$ and $\mathcal{F}(\ini_{\tau} (I))$, respectively.

Let $\calA$ be a vector configuration of $2$ integer vectors $(1,0),(0,1) \in \mathbb{Z}^2$. Denote by $\mathbb{N}\calA$ the affine monoid generated by $\calA$. 
Let $\mu$ and $r$ be two positive integers. Consider the polynomial rings $\KK[\bfX,T_1,\dots,T_{\mu}]\coloneqq \mathbb{K}[x_{1,1},\dots,x_{n,m}, T_1,\ldots, T_{\mu}]$, and $\mathbb{K}[y,t_1,\ldots, t_r]$,
where $\mathbb{K}$ is a field. Define an $\mathbb{N}\calA$-grading on 
them
via 
\[
    \deg(x_{i,j}) = \deg(y) =(1,0) 
\]
for all $i\in [n]$ and $j\in [m]$, and 
\[
    \deg(T_i) = \deg(t_j) =(0,1).
\]
for all $i\in [\mu]$ and $j\in [r]$.
Consider an additional ring 
\[
    \mathbb{K}[\bfZ, \bfT] \coloneqq \KK[z_{1,1},\dots,z_{n,m},T_{1,1},\dots,T_{\mu,r}]
\]
with grading $\deg(z_{i,j}) = (1,0)$, $\deg(T_{i,j}) = (0,1)$, and let 
\[
    \varphi : \mathbb{K}[\bfZ, \bfT]  \rightarrow \mathbb{K}[\bfX, T_1,\ldots, T_{\mu}] \otimes \mathbb{K}[y,t_1,\ldots, t_r] 
\]
be the ring homomorphism defined by 
$\varphi(z_{i,j}) = x_{i,j}\otimes y$ and $\varphi(T_{i,j}) = T_i\otimes t_j$.

\begin{Definition}
    Let $K_1 \subseteq \mathbb{K}[\bfX, T_1,\ldots, T_{\mu}]$ and  $K_2 \subseteq\mathbb{K}[y,t_1,\ldots, t_r]$ be two $\mathbb{N}\calA$-homogeneous ideals, and denote by
    \[
        \pi_{K_1} : \mathbb{K}[\bfX, T_1,\ldots, T_{\mu}] \rightarrow A:=\mathbb{K}[\bfX, T_1,\ldots, T_{\mu}]/K_1
    \]
    and 
    \[
        \pi_{K_2} : \mathbb{K}[y,t_1,\ldots, t_r]  \rightarrow B:=\mathbb{K}[y,t_1,\ldots, t_r] /K_2
    \]
    the canonical projections. Their tensor product as $\mathbb{K}$-algebra homomorphisms is $\pi_{K_1} \otimes \pi_{K_2}$.
    The \emph{toric fiber product} of $K_1$ and $K_2$ is the ideal 
    \[
        K_1 \times_{\calA} K_2 := \ker\left( (\pi_{K_1} \otimes \pi_{K_2}) \circ \varphi \right).
    \]
    The \emph{\textup{(}multigraded\textup{)} Segre product} of $A$ and $B$ is 
    \[
        A\#_{\mathbb{N}{\calA}} B := \bigoplus_{a \in \mathbb{N}\calA} A_a \otimes B_a.
    \]  
\end{Definition}

\begin{Proposition}
    \label{fibertoRees}
    Adopting the notations established above, we have
    \[
        \mathbb{K}[\bfZ,\bfT]/(\calJ^{\ast} \times_{\calA} 0) \cong \calR(\ini_{\tau}(I))\#_{\mathbb{N}\calA} \mathbb{K}[y,t_1,\ldots, t_r] \cong \calR(\oplus _{i=1}^r \ini_{\tau}(I)),
    \]
    where $\calJ^{\ast}$ is the presentation ideal of $\calR(\ini_{\tau}(I))$.
\end{Proposition}
\begin{proof}
    We can set $K_1=\calJ^{\ast}$
    and $K_2=0$. The first isomorphism follows from 
    \cite[Lemma 7]{KRSegre}.
    As for the second isomorphism, let $J\coloneqq \ini_{\tau}(I)$. We have the natural bi-grading on $\calR(J)=\KK[\bfX, \ini_{\tau}(f_i)t \mid i\in [\mu]]$ where $\deg (x_{i,j})=(1,0)$ and $\deg (t)=(-d,1)$. Likewise, we have the bi-grading on  $\mathbb{K}[y,t_1,\ldots, t_r]$ with $\deg (y)=(1,0)$ and $\deg (t_i)=(0,1)$. 
    The degree $(a,b)$ component of $\calR(J)\#_{\mathbb{N}\calA} \mathbb{K}[y,t_1,\ldots, t_r]$ is generated over $\KK$ by elements of the form
    \[
        g_a
        \prod _{i=1}^{\mu}(\ini_{\tau}(f_i) t)^{c_i}\otimes y^a \prod_{j=1}^r t_j^{b_j}
    \]
    where $g_a$ is a monomial in $\KK[\bfX]$ of degree $a$.
    and $\sum_{i=1}^{\mu}c_i=\sum _{j=1}^r b_j=b$. On the other hand, 
    $\calR(\oplus _{i=1}^r J)=\KK[\bfX, \ini_{\tau}(f_i)t_j \mid i\in [\mu], j\in [r]]$
    is bigraded with $\deg (x_{i,j})=(1,0)$ and $\deg (t_j)=(-d,1)$.
    The degree $(a,b)$ component of $\calR(\oplus_{i=1}^r J)$
    is generated over $\KK$ by elements of the form
    \[
        g_a
        \prod _{i=1}^{\mu}\prod_{j=1}^r(\ini_{\tau}(f_i) t_j)^{b_{i,j}}.
    \] 
    where $g_a$ is a monomial in $\KK[\bfX]$ of degree $a$,
    and $\sum_{i=1}^{\mu}\sum_{j=1}^rb_{i,j}=b$. The second claimed isomorphism follows naturally.
\end{proof}

\begin{Definition}
    \label{EQForMulti-ReesInitial}
    Let \( \calA \) be the simple vector configuration above
    and $\calJ^{\ast}\subset \mathbb{K}[\bfX, T_1,\ldots, T_{\mu}]$ be the presentation ideal of $\calR(\ini_{\tau}(I))$.
    \begin{enumerate}[a]
        \item Let \( G \subset \calJ^{\ast} \) be a collection of bi-homogeneous polynomials. To each 
            \[
                g=\sum_{p=1}^{q} \lambda_p g_{p}  \prod _{k=1}^b T_{i_k} \in G 
            \]
            of bi-degree $(a,b)$, where $g_{p}$ is a monomial element in $\KK[\bfX]$ of degree $a$ and $ \prod _{k=1}^b T_{i_k}$ is a monomial in $\KK[T_i \mid i\in [\mu]]$ of degree $b$, we associate with the set
            \[
                \Lift(g) \coloneqq \left\{ \sum_{p=1}^{q} \lambda_p \widetilde{g_{p}}  \prod _{k=1}^b T_{i_k, j}\;\Big|\; j \in [r] \right\}\subset \calJ^{\ast} \times_{\calA} 0,
            \]
            where $\widetilde{g_{p}}=h(g_p)$ for the natural map $h: \KK[\bfX] \rightarrow \KK[\bfZ]$ such that $h(x_{i,j})=z_{i,j}$ for all $i$ and $j$.
            The set $\Lift(G)\coloneqq\bigcup_{g\in G}\Lift(g)$ will be called 
            the \emph{lifting} of \( G \) to \( \calJ^{\ast} \times_{\calA} 0 \).

        \item 
            Let $\calC=\KK[T_{i}t_j \mid i\in [\mu], j\in [r]]$ be a polynomial ring and $\bfC=(T_{i,j})_{i\in [\mu], j\in [r]}$ be a $\mu \times r$ matrix of indeterminates. Denote the set of $2$-minors of $\bfC$ by $\Quad_{\bfC}$. Mapping $T_{i,j}$ to $T_it_j$ yields a presentation of $\calC\cong \KK[T_{i,j}]/I_2(\bfC)$ where $I_2(\bfC)$ is the ideal generated by $\Quad_{\bfC}$.
    \end{enumerate}
\end{Definition}

\begin{Theorem}
    \label{GBMultiRessInitial_FiberProduct}
    Adopt the setting in \Cref{EQForMulti-ReesInitial}, let $G$ be a Gr\"obner basis of the presentation ideal \( \calJ^{\ast}\) of $\calR(\ini_{\tau}(I))$ in $\mathbb{K}[\bfX, T_1,\ldots, T_{\mu}]$ with respect to a weight order $\mathbf{w}_1$.
    Then $\Lift (G)\cup 
    \Quad_{\bfC} $ is a Gr\"obner basis of \( \calJ^{\ast} \times_{\calA} 0 \) in   $\mathbb{K}[
    \bfZ,
    T_{i,j} \mid i\in [\mu], j\in [r]]$ with respect to some weight order $\bfw$ associated with $\bfw_1$.
    In particular, we have
    \[
        \calR(\oplus_{i=1}^r \ini_{\tau}(I))\cong \mathbb{K}[
        \bfZ,
        T_{i,j} \mid i\in [\mu], j\in [r]]/(\Lift(G)\cup 
        \Quad_{\bfC} ).
    \]
    Furthermore, if $\calJ^{*}$ has a squarefree initial, then so does $\calJ^{*}\times_{\calA} 0$.
\end{Theorem}
\begin{proof}
    The Gr\"obner basis statement is by  \cite[Theorem 3.3]{Shibuta} and \cite[Corollary 14]{SullivantToricFiber}. 
    The isomorphism in the ``in particular'' part is established by \Cref{fibertoRees}. The final squarefree statement follows from \cite[Corollary 15]{SullivantToricFiber}.
\end{proof}

\subsection{SAGBI Bases}
\label{sec:sagbi}

Our goal of this work is to utilize the theory of \textsc{Sagbi} deformations for Rees algebras, as developed in \cite{CHV96},
to determine the presentation ideals of the multi-Rees algebras and associated special fiber rings of determinantal ideals. With 
\cite[Lemma 7]{KRSegre},
we can focus on finding the presentation ideal of $\calR(\ini (I))$, the Rees algebra of the initial ideal of $I$; see also \Cref{fibertoRees}. This approach allows for leveraging the Rees algebra of the initial ideal to gain insights into the Rees algebra of the given ideal. 

In this subsection, we will only recall the definition of a \textsc{Sagbi} basis for reference. For more information on \textsc{Sagbi} basis theory, see \cite[Chapter 11]{Sturmfels}. For detailed applications of \textsc{Sagbi} bases to Rees algebras, see \cite{CHV96,BCresPowers}.

\begin{Definition}
    Let $P$ be a polynomial ring over the field $\KK$ endowed with a monomial order $\sigma$. Suppose that $A$ is a finitely generated $\KK$-subalgebra of the ring $P$. 
    The \emph{initial algebra} of $A$ with respect to $\sigma$, denoted by $\ini_\sigma(A)$, is the $\mathbb{K}$-subalgebra of $P$ generated by the initial monomials $\ini_\sigma(a)$ where $a\in A$. 
    Moreover, a set of elements $\calA\subseteq A$ is called a \emph{\textsc{Sagbi} basis} if $\ini_\sigma(A) = \KK[\ini_\sigma(a)\mid a\in \calA]$.
\end{Definition}

Next, we fix a term order on the ring $R=\KK[\bfX]$ introduced in the previous subsection.

\begin{Setting}
    \label{lexOrder}
    \begin{enumerate}[a]
        \item Let the entries in $\bfX$ (i.e., the variables of $R$) be ordered, so that $x_{i,j}>x_{l,k}$ if and only if $i<l$ or $i=l$ and $j<k$.
        \item Let $\tau$ be the lexicographical order on the ring $R=\KK[\bfX]$ with respect to the aforementioned total order $>$ on the variables of $R$.
           
        \item Let $\delta$ be the graded reverse lexicographical order on $\KK[t_1,\ldots,t_r]$ induced by $t_1>t_2>\cdots>t_r$.
        \item We 
            extend  the lexicographic order $\tau$
            to a monomial order $\tau'$ on $R[t_1,\ldots,t_r]$ as follows: for monomials $m_1\cdot \prod_{i=1}^r t_i^{j_i}$ and $m_2\cdot \prod_{i=1}^r t_i^{k_i}$ of $R[t_1,\ldots,t_r]$, set $m_1\cdot\prod_{i=1}^r t_i^{j_i}>_{\tau'} m_2\cdot\prod_{i=1}^r t_i^{k_i}$ if $\prod_{i=1}^r t_i^{j_i}>_{\delta} \prod_{i=1}^r t_i^{k_i}$  or if  $\prod_{i=1}^r t_i^{j_i}= \prod_{i=1}^r t_i^{k_i}$ and $m_1>_{\tau}m_2$ in $R$. 
    \end{enumerate} 
\end{Setting}

It is desirable
to use the \textsc{Sagbi} deformation techniques developed in \cite{CHV96} to study $\mathcal{R}(M)$. 
Since
\[ 
    (\ini_{\tau}(I))^{a_1}\cdots (\ini_{\tau}(I))^{a_r} \subseteq \ini_{\tau} (I^{a_1} \cdots I^{a_r}),
\]
we have 
\begin{equation}
    \mathcal{R}(\ini_{\tau} (I) \oplus \cdots \oplus \ini_{\tau} (I)) \subseteq \ini_{\tau'}(\mathcal{R}( I \oplus \cdots \oplus I)). 
    \label{eqn:rees_initial_direct_sum}
\end{equation}

Therefore, it is natural to gain an understanding of $\mathcal{R}( \bigoplus_{i=1}^r\ini_{\tau} (I))$ on the left-hand side of \eqref{eqn:rees_initial_direct_sum}
first. According to \Cref{GBMultiRessInitial_FiberProduct}, we will focus on $\mathcal{R}(\ini_{\tau} (I)$ from now on.

\subsection{Distributive lattice and Hibi ring}
\label{subsection:Hibi}
Given that our treatment draws upon the Hibi ring theory concerning distributive lattices, it is necessary to elucidate the pertinent terminology.
Let $\sL$ be a finite lattice. By definition, it is a poset with respect to a partial order $\le$. Furthermore, for any two elements $a,b\in \sL$, there is a unique greatest lower bound $a\wedge b$, called the \emph{meet} of $a$ and $b$, and there is a unique least upper bound $a\vee b$, called the \emph{join} of $a$ and $b$. 
A subposet $\sL'$ of $\sL$ is called a \emph{sublattice} if $\sL'$ itself is a lattice, and for any $a,b\in\sL'$, the meet and join of $a$ and $b$ in $\sL'$ coincide with those in $\sL$, respectively. 

Suppose that the join and meet operations of $\sL$ satisfy
\[
    (a\wedge b)\vee (a\wedge c) = a\wedge (b\vee c) \quad \text{and}\quad
    (a\vee b)\wedge (a\vee c) = a\vee (b\wedge c)
\]
for all $a,b,c\in \sL$, then $\sL$ is called a \emph{distributive lattice}.

Given a distributive lattice $\sL$, we can consider the polynomial (known as the \emph{Hibi relation})
\[
    f_{a,b}\coloneqq T_aT_b-T_{a\wedge b}T_{a\vee b}
\]
in the polynomial ring $\KK[T_c\mid c\in \sL]$ over the field $\KK$.
The ideal $I_{\sL}\coloneqq \braket{f_{a,b}\mid a,b\in \sL}$ is called the \emph{join-meet ideal} or \emph{Hibi ideal} associated with $\sL$. 

Let $<$ be a monomial order on $\KK[T_c\mid c\in \sL]$. If for all incomparable $a, b \in \sL$, one has $\ini_<(f_{a,b})= T_aT_b$, the monomial order $<$ is called \emph{compatible} in the language of \cite{MR3838370}. Let $\calG_{\sL}$ be the collection of all $f_{a,b}$ such that $a$ and $b$ are incomparable in $\sL$. Then $\calG_{\sL}$ is a Gr\"obner basis of $I_{\sL}$ by \cite[Theorem 6.17]{MR3838370} with respect to any compatible monomial order.

Furthermore, $I_{\sL}$ is a prime ideal. The quotient ring $\KK[\sL]\coloneqq \KK[T_c\mid c\in \sL]/I_{\sL}$ is called the \emph{Hibi ring} associated with $\sL$. It is known that $\KK[\sL]$ is a toric ring, by \cite[Theorem 6.19]{MR3838370}.

\section{Initial algebra of the multi-fiber ring of generic matrices}

The presentation ideal of the fiber ring for the ideal generated by the maximal minors of a generic matrix is well-known to be generated by the Pl\"ucker relations. This fiber ring serves as the coordinate ring of the Grassmannian variety. Additionally, the maximal minors constitute a \textsc{Sagbi} basis of the fiber ring, as detailed in \cite[Theorem 6.46]{EHgbBook}. Specifically, a Gr\"obner basis of the presentation ideal for the fiber ring of the initial ideal of the maximal minors can be lifted to a Gr\"obner basis of the presentation ideal for the fiber ring of the maximal minors, generated by Pl\"ucker relations. Motivated by this, our goal is to determine a Gr\"obner basis of the presentation ideal for the fiber ring of the direct sum of the initial ideal of the maximal minors.

As in \Cref{sec:prelim}, $\bfX = (x_{i,j})$ is an $n\times m$ generic matrix over the field $\KK$, where $n\leq m$. Meanwhile, $R=\KK[\bfX]$ is the associated polynomial ring, equipped with the term order $\tau$, as delineated in \Cref{lexOrder}.

Let $I_n(\bfX)$ be the ideal in $R$, generated by the maximal minors of $\bfX$. It is well-known that the set of maximal minors is a Gr\"obner basis of $I_n(\bfX)$ with respect to any diagonal monomial order, by \cite[Theorem 6.33]{EHgbBook}. 
To take advantage of this information,
let
\[
    \calD\coloneqq \{\bda=(a_1,\ldots, a_{n}) \in \ZZ_+^{n}\mid a_1 <\cdots < a_n \leq m\},
\]
and we write 
\[
    [\bda] \coloneqq \det[x_{i,a_j} \mid i,j\in [n]]
\]
for the maximal minor of $\bfX$ with columns indexed by the tuple $\bda=(a_1,\dots,a_n) \in \calD$.
Since $\tau$ is a diagonal monomial order, we have
\[
    \ini_{\tau}(I_n(\bfX))=\braket{\bfx_\bda \coloneqq x_{1,a_1}\cdots x_{n,a_n}\mid  \bda\in \calD}.
\]

There is a natural poset structure on $\calD$, given by
\[
    (a_1,\dots,a_n)\le (b_1,\dots,b_n) \qquad \Leftrightarrow \qquad a_i\le b_i\quad\text{for all $i$.}
\]
Under this partial ordering, $\calD$ becomes a lattice. 
In fact, the \emph{join} is given by 
\[
    (a_1,\dots,a_n)\vee (b_1,\dots,b_n) = (\max\{a_1,b_1\},\dots,\max\{a_n,b_n\}),
\]
while the \emph{meet} is given by 
\[
    (a_1,\dots,a_n)\wedge (b_1,\dots,b_n) = (\min\{a_1,b_1\},\dots,\min\{a_n,b_n\}).
\]

It is obvious that $\calD$ becomes a distributive lattice with respect to the above meet and join operations.

By abuse of notation, we will also consider $[r] \coloneqq \{1,2,\dots,r\}$ to be a linearly ordered set with respect to the natural comparison. It is a distributive lattice, where $i\vee j=\max\{i,j\}$ and $i\wedge j=\min\{i,j\}$ for all $i$ and $j$ in $[r]$.
Furthermore, the Cartesian product lattice of $\calD$ and $[r]$, $\calD\times [r]$, is still distributive. 

As mentioned in the previous section, a Hibi ring is a toric ring. However, it is sometimes too naive to expect that endowing a finite set of integral vectors $A$ with a distributive lattice structure will automatically result in an isomorphism between the associated Hibi ring and the semigroup ring $\KK[A]$. Therefore, we must extend \cite[Theorem 6.45]{EHgbBook} to the upcoming \Cref{CanUseRestriction}, since we intend to apply Hibi ring theory to our toric rings.

\begin{Proposition}
    \label{FiberofOneGen}
    The toric ring $\KK[\bfx_{\bda}t_i\mid 
    (\bda,i)\in \calD\times[r]
    ]$ is isomorphic to the Hibi ring $\KK[\calD\times [r]]$.
    In particular, the join-meet ideal $I_{\calD\times [r]}$ is the toric ideal of the toric ring $\KK[\bfx_{\bda}t_i\mid 
    (\bda,i)\in \calD\times[r]
    ]$ in $\KK[T_{\bda,i}\mid 
    (\bda,i)\in\calD\times [r]
    ]$. The set $\calG_{\calD\times[r]}$ of nontrivial Hibi relations 
    is a Gr\"obner basis of this toric ideal with respect to compatible monomial orders.
\end{Proposition}
\begin{proof}
    The ring map 
    \[
        \psi^\ast:\KK[T_{\bda,i}\mid 
        (\bda,i)\in\calD\times[r]
        ]\to \KK[\bfx_{\bda}t_i\mid 
        (\bda,i)\in \calD\times[r]
        ], \quad T_{\bda,i}\mapsto \bfx_{\bda}t_i
    \]
    is clearly surjective. Furthermore, the join-meet ideal $I_{\calD\times[r]}$ is a subideal of $\ker(\psi^\ast)$. Since both the toric ring $\KK[\bfx_{\bda}t_i\mid \bda\in \calD,i\in [r]]$ and the Hibi ring $\KK[\calD\times [r]]$ are integral domains, it suffices to show that the Krull dimensions of these two rings are the same.

    On one hand, since $\KK[\bfx_{\bda}t_i\mid \bda\in \calD,i\in [r]]$ is a finitely generated integral domain over the field $\KK$, we have
    \begin{align*}
        \dim (\KK[\bfx_{\bda}t_i\mid
        (\bda,i)\in \calD\times[r]
        ]) & =\trdeg_{\KK}(\KK[\bfx_\bfc t_i\mid
        (\bfc,i)\in\calD\times [r]
        ])\\
        &=\trdeg_{\KK}(\KK[\bfx_\bfc t_1\mid\bfc\in\calD])+(r-1)\\
        &=\dim(\KK[\bfx_\bfc\mid\bfc\in\calD])+(r-1)\\
        &=n(m-n)+1+(r-1),
    \end{align*}
    where the last equality follows from \cite[Theorem 6.45]{EHgbBook}.

    On the other hand, it follows from \cite[Lemma 6.44]{EHgbBook} that the join irreducible elements of $\calD$ are $\bdc_{ik} \coloneqq (1,2,\dots,i,i+1+k,i+2+k,\dots,n+k)$ with $i = 0,1,\dots,n-1$ and $k = 1,\dots,m-n$. Hence one can easily check that the join irreducible elements of $\calD\times [r]$ are $\bdc_{ik}' \coloneqq (1,2,\dots,i,i+1+k,i+2+k,\dots,n+k,1)$ and $\bda_{n\ell}\coloneqq (1,\ldots,n,\ell)$ with $i = 0,1,\dots,n-1$, $k = 1,\dots,m-n$, and $\ell\in [2,r]$. It follows from \cite[Theorem 6.38]{MR3838370} that $\dim(\KK[\calD\times[r]])= n(m-n)+(r-1)+1$. Therefore, the two integral domains have the same Krull dimension, as desired.
\end{proof}  

From now on, we fix a compatible monomial order $<$ on $\KK[T_{\bda,i}\mid 
(\bda,i)\in\calD\times[r]]$.

\begin{Proposition}
    \label{CanUseRestriction}
    Let $\calD'$ be a nonempty sublattice of $\calD$ of $\calD'$.
    Then the set $\calG_{\calD'\times[r]}$ of nontrivial Hibi relations 
    in the polynomial ring $\KK[T_{\bda,i}\mid 
    (\bda,i)\in \calD'\times[r]
    ]$ is a Gr\"obner basis of the toric ideal of the toric ring $\KK[\bfx_{\bda}t_i\mid
    (\bda,i)\in \calD'\times[r]
    ]$ with respect to the monomial order induced from $<$.
    In particular, the join-meet ideal $I_{\calD'\times [r]}$
    in $\KK[T_{\bda,i}\mid 
    (\bda,i)\in\calD'\times[r]
    ]$ is the toric ideal of the toric ring $\KK[\bfx_{\bda}t_i\mid
    (\bda,i)\in \calD'\times[r]
    ]$, and the toric subring $\KK[\bfx_{\bda}t_i\mid
    (\bda,i)\in \calD'\times[r]
    ]$ is isomorphic to the Hibi ring 
    $\KK[\calD'\times [r]]$.
\end{Proposition}
\begin{proof}
    Let $\calG_{\calD\times[r]}$ be the set of nontrivial Hibi relations associated to the product lattice $\calD\times [r]$. We have seen in \Cref{FiberofOneGen} that $\calG_{\calD\times[r]}$ is a Gr\"obner basis of the toric ideal $I_{\calD\times[r]} \subset \KK[T_{\bda,i}\mid
    (\bda,i)\in\calD\times[r]
    ]$ of the toric ring $\KK[\bfx_\bda t_i\mid 
    (\bda,i)\in\calD\times[r]
    ]$, with respect to the compatible monomial order $<$. 

    Every $f\in \calG_{\calD\times [r]}$ takes the form 
    \[
        f_{(\bda,i),(\bdb,j)} \coloneqq T_{\bda,i}T_{\bdb,j}-T_{(\bda,i)\vee(\bdb,j)}T_{(\bda,i)\wedge(\bdb,j)}
    \]
    for non-comparable $(\bda,i),(\bdb,j)\in \calD\times [r]$. If $\ini(f_{(\bda,i),(\bdb,j)})=T_{\bda,i}T_{\bdb,j}\in \KK[T_{\bdc,k}\mid 
    (\bdc,k) \in \calD'\times[r]
    ]$, it is clear that $\bda,\bdb\in \calD'$. Therefore, $\bda\wedge\bdb,\bda\vee\bdb\in \calD'$ by our assumption. This implies that $f\in \KK[T_{\bdc,k}\mid
    (\bdc,k) \in \calD'\times [r]
    ]$. Therefore, it follows from \cite[Proposition 2.1]{DeNegri} that $\calG_{\calD'\times[r]}$ is a Gr\"obner basis of the toric ideal $I_{\calD'\times[r]} \subset \KK[T_{\bda,i}\mid
    (\bda,i)\in\calD'\times[r]
    ]$ of the toric ring $\KK[\bfx_\bda t_i\mid 
    (\bda,i)\in\calD'\times[r]
    ]$, with respect to the monomial order induced from $<$.
    It remains to notice the trivial fact that $\calG_{\calD'\times [r]}$ generates the join-meet ideal $I_{\calD'\times[r]}$.
\end{proof}

\section{Multi-Rees algebras of the initial of ideal of maximal minors of ladder matrices}

The ensuing next two sections are devoted to the examination of multi-Rees algebra associated with the maximal minors of a ladder matrix, thereby extending the findings presented
in \cite{CDFGLPS}.

In this preliminary section, the focus is on the algebra defined by the initials. 
For the sake of reference, 
$\mathbf{X}$ is an $n \times m$ generic matrix with the assumption $n \leq m$.

\begin{Definition}
    \label{def:Ladder}
    A subset $\bfL$ of $\bfX$ is called a \emph{ladder} if it has the following property: whenever $x_{i,j},x_{i',j'} \in \bfL$ and $i \le i'$, $j \ge j'$, then $x_{u,v}\in \bfL$ for any pair $(u,v)$ with $i \le u \le i'$ and $j' \le v\le  j$. In other words, a ladder has the property that whenever the upper right and the lower left corners of a submatrix of $\bfX$ belong to the ladder, then all the indeterminates in the submatrix belong to the ladder. 
    As a result, there are integers
    \begin{align*}
        1=\ell_1\le \ell_2 \le \cdots \le  \ell_n\le m
        \intertext{and}
        1\le k_1\leq  k_2 \leq \cdots \leq k_n=m,
    \end{align*}
    such that the element $x_{i,j}$ of $\bfX$ belongs to $\bfL$ if and only if $\ell_i\le j\le k_i$.

    Given a ladder $\bfL$, there are two types of minor ideals associated with it.
    \begin{enumerate}
        \item Let $I_t(\bfL)$ be the ideal of $\KK[\bfL]$ generated by all the $t$-minors of $\bfX$ involving only indeterminates of $\bfL$. This ideal is known as a \emph{ladder determinantal ideal} in the literature. The \emph{ladder determinantal ring} $\KK[\bfL]/I_t(\bfL)$ has been well studied; see, for instance, 
            \cite{MR926272,MR1319965,MR991408}. 
            It is known that this ring is a Cohen--Macaulay normal domain. 
        \item Meanwhile, let $\bfX_{\bfL}$ be the matrix obtained from $\bfX$ by setting to zero all the elements not in $\bfL$.
            We will call this matrix the \emph{ladder matrix} associated with $\bfL$.  Let $I_t(\bfX_{\bfL})$ be the ideal of $\KK[\bfL]$ generated by all the $t$-minors of $\bfX_{\bfL}$. In general, the ideal $I_t(\bfX_{\bfL})$ is not prime. 
    \end{enumerate}
\end{Definition}

In this paper, we are only interested in the ideal $L=I_n(\bfX_{\bfL})$, which is generated by the maximal minors of $\bfX_{\bfL}$. We will demonstrate that its multi-Rees algebra exhibits favorable properties. For $\bda=(a_1,\dots,a_n)\in \calD$, we write $[\bda]$ for the maximal minor of $\bfX$ spanned by the columns corresponding to the indices $a_1,\dots,a_n$. Similarly, we will let $[\bda]_{\bfL}$ denote the maximal minor of $\bfX_{\bfL}$.

\begin{Example}
    \label{LadderMatrix}
    The $3\times 9$ matrix
    \[
        \bfX_{\bfL}=\begin{bNiceMatrix}
            x_{1,1}  & x_{1,2} &x_{1,3} & x_{1,4} &x_{1,5} & x_{1,6} &0 &0&0\\
            0 & 0  &x_{2,3} &x_{2,4} &x_{2,5} &x_{2,6} & 0 &0&0\\
            0  &0  &0 &x_{3,4} &x_{3,5} &x_{3,6}&x_{3,7}&x_{3,8}&x_{3,9}
            \CodeAfter
            \tikz \draw [blue] (1-|1) -| (2-|1) -| (3-|3) -| (4-|4) -| (4-|10) -| (3-|10) -| (1-|7) -| cycle;
        \end{bNiceMatrix}
    \]
    is a ladder matrix, associated with the ladder 
    \[
        \bfL=\{x_{1,1} ,  x_{1,2} , x_{1,3} ,  x_{1,4} , x_{1,5} ,  x_{1,6} , x_{2,3} , x_{2,4} , x_{2,5} , x_{2,6} , x_{3,4} , x_{3,5} , x_{3,6}, x_{3,7}, x_{3,8}, x_{3,9}\}. 
    \]
    The ladder determinantal ideal $I_3(\bfL)$ is the principal ideal generated by 
    \[
        [(4,5,6)]=
        \det \begin{pmatrix}
            x_{1,4} & x_{1,5} & x_{1,6} \\
            x_{2,4} & x_{2,5} & x_{2,6} \\
            x_{3,4} & x_{3,5} & x_{3,6} 
        \end{pmatrix}=[(4,5,6)]_{\bfL}.
    \]
    On the other hand, the 
    determinantal ideal $L=I_3(\bfX_{\bfL})$ has many more generators, with the monomial $[(1,3,4)]_{\bfL}=x_{1,1}x_{2,3}x_{3,4}$ being one of them.
\end{Example}

Since some of the maximal minors of $\bfX_{\bfL}$ are zero, we need to keep track of the nonzero ones.
Let $\calL\coloneqq \{\bda\in \calD \mid [\bda]_{\bfL}\ne 0\}$, and call it the \emph{associated column index set}.

\begin{Lemma}
    \label{detNotzero}
    Suppose that $\bda=(a_1,\dots,a_n)\in \calD$. Then $[\bda]_{\bfL}\ne 0$ if and only if the $x_{i,a_i}\in \bfL$ for all $i\in [n]$.
\end{Lemma}
\begin{proof}
    The monomial $x_{1a_1}x_{2a_2}\cdots x_{na_n}$ is a term in the full expansion of the determinant $[\bda]$. If $x_{i,a_i}\in \bfL$ for all $i\in [n]$, this monomial is also a nonzero term of $[\bda]_{\bfL}$. In particular, $[\bda]_{\bfL}\ne 0$.

    On the other hand, if the $(i_0,a_{i_0})$ entry of $\bfX_{\bfL}$ is zero for some $i_0\in[n]$, then either $a_{i_0}<\ell_{i_0}$ or $a_{i_0}>k_{i_0}$. In the first case, the $(i,j)$ entry of $\bfX_{\bfL}$ is $0$ for all $i\ge i_0$ and $j\le a_{i_0}$. Using the Laplace cofactor expansion from linear algebra, it is easy to see that $[\bda]_{\bfL}=0$. The second case is symmetric and similar.
\end{proof}

From this point forward, we will assume that the associated column index set $\calL$ is nonempty. First of all, we will show that $\calL$ inherits the lattice structure of $\calD$.

\begin{Lemma}
    \label{lem:LexClosed}
    Under the componentwise comparison, $\calL$ becomes a sublattice of $\calD$.
\end{Lemma}
\begin{proof}
    In the definition of the ladder $\bfL$, we have the left boundary $\bdl=(\ell_1,\dots,\ell_n)$ and the right boundary $\bdk=(k_1,\dots,k_n)$ in $\ZZ_+^n$.  Notice that $\ZZ_+^n$ is a distributive lattice with respect to the componentwise comparison, and $\calD$ is a sublattice of it. 
    Let $\bdl^\dag\coloneqq \bigwedge\{\bda\in\calD\mid \bdl\le \bda\}$ and $\bdk^\dag\coloneqq\bigvee\{\bda\in\calD\mid \bda\le \bdk\}$.
    It is readily seen that $\bdl^\dag,\bdk^\dag\in \calD$, and $\calL$ is precisely the closed interval $[\bdl^\dag,\bdk^\dag]\coloneqq\{\bda\in\calD\mid \bdl^\dag\le \bda\le \bdk^{\dag}\}$ of the lattice $\calD$. 
\end{proof}

\begin{Remark}
    Suppose that $\bda,\bdb\in \calD$ satisfy $\bda\vee\bdb, \bda\wedge\bdb\in [\bdl^\dag,\bdk^\dag]$ for the $\bdl^\dag$ and $\bdk^\dag$ introduced in the proof of \Cref{lem:LexClosed}. Since $\bda\wedge\bdb\le \bda\le \bda\vee\bdb$, this implies $\bda\in [\bdl^\dag,\bdk^\dag]$. Similarly, we also have $\bdb\in \bdk^\dag$. In other words, $\calL$ is an \emph{embedded sublattice} in the sense of \cite[Definition 3.9]{MR2761126}. As an interesting consequence of this fact, the toric variety associated with $\calL$ is a subvariety of the toric variety associated with $\calD$, by \cite[Proposition 3.10]{MR2761126}. Likewise, $\calL\times[r]$ is an embedded sublattice of $\calD\times[r]$.
\end{Remark}

Similar to the generic matrix $\mathbf{X}$, the canonical generators in the set $\{[\boldsymbol{\alpha}]_{\mathbf{L}} \mid \boldsymbol{\alpha} \in \mathcal{L}\}$ form a Gr\"obner basis.

\begin{Lemma} 
    \label{LadderSet}
    The nonzero maximal minors of a ladder matrix $\bfX_{\bfL}$ form a minimal generating set for $L=I_n(\bfX_{\bfL})$ and 
    serve as a universal Gr\"{o}bner basis for the determinantal ideal $L$.
\end{Lemma}
\begin{proof}
    Notice that the ladder matrix $\bfX_\bfL$ is a sparse matrix. Therefore, it follows from \cite[Theorem 4.1]{Boocher} that $L$ has a linear free resolution arising from a specialization of the Eagon--Northcott complex. In particular, nonzero maximal minors of $\bfX_\bfL$ form a minimal generating set of $L$.
    The remaining universal Gr\"{o}bner basis property follows from \cite[Proposition 5.4]{Boocher}.
\end{proof}

Clearly, $I_n(\bfX)$ and 
$I_n(\bfX_{\bfL})$ resemble each other not only in the context of the Gr\"obner basis, but also in their lattice structures.
Therefore, it makes sense to use the restriction apparatus developed in the previous section. 
Recall that we introduced the lexicographic monomial order $\tau$ in \Cref{lexOrder} on the ring $\KK[\bfX]$. By abuse of notation, we also denote its restriction to the subring $\KK[\bfL]$ by $\tau$. In particular, for the ideal $L=I_n(\bfX_\bfL)$, its initial ideal is
\[
    \ini_\tau(L)=\braket{\bfx_\bda\mid \bda\in \calL}
\]
by \Cref{LadderSet}.
Now, we describe the defining equations of $\calR(\bigoplus_{i=1}^r \ini_{\tau}(L))$ and $\calF(\bigoplus_{i=1}^r \ini_{\tau}(L))$.

\begin{Definition}
    Suppose that $\widetilde{\bdc} = \{c_1, \ldots, c_{n+1}\}$ is a set of integers such that $1\leq c_1<c_2<\cdots<c_{n+1}\le m$. For each $i\in [n+1]$, we denote the tuple $(c_1,\dots,c_{i-1},c_{i+1},\dots,c_{n+1})$ in $\calD$ by $\widetilde{\bdc}\setminus c_i$. If both $\widetilde{\bdc}\setminus c_i$ and $\widetilde{\bdc}\setminus c_{i+1}$ are elements of $\calL$, then by Eagon--Northcott complex, we have an \emph{Eagon--Northcott type} binomial relation (with respect to $\calL$):
    \begin{equation}\label{eqn:ENT}
        x_{i,c_i}T_{\widetilde{\bdc}\setminus c_i}-x_{i,c_{i+1}}T_{\widetilde{\bdc}\setminus c_{i+1}}
    \end{equation}
    in the ambient ring $R[T_{\bda}\mid \bda\in\calL]$
    where $R=\KK[\bfL]$.
    Let $H$ to be the collection of Eagon--Northcott type binomials. 
\end{Definition}

The following is one of the main theorems of this work.

\begin{Theorem}
    \label{FiberLadder}
    The set $\calG_{\calL\times[r]}$ of nontrivial Hibi relations associated with the distributive lattice $\calL\times[r]$
    constitutes a Gr\"{o}bner basis of the presentation ideal of the special fiber ring $\calF\left(\bigoplus_{i=1}^{r} \ini_{\tau} (L)\right)$ 
    with respect to any compatible order $\sigma$ in $\KK[T_{\bda,i} \mid \bda\in \calL\times [r]]$.
    In particular, the $\KK$-algebra $\calF\left(\bigoplus_{i=1}^{r} \ini_{\tau} (L)\right)$ is a Hibi ring.
\end{Theorem}

\begin{proof}
    Since $\calL$ is a distributive sublattice of $\calD$ by \Cref{lem:LexClosed}, \Cref{CanUseRestriction} can be applied directly to the product lattice $\calL\times[r]$. 
\end{proof}

\begin{Remark}\label{GWith2minors}
    Let $\bfC=(T_{\bda,j})_{\bda\in \calL, j\in [r]}$ be an $|\calL| \times r$ matrix of indeterminates, and denote the set of $2$-minors of $\bfC$ by $\Quad_{\bfC}$.  We can compute the toric fiber product as in \Cref{GBMultiRessInitial_FiberProduct}, which shows that $\Lift(\mathcal{G}_{\mathcal{L}}) \cup \Quad_{\mathbf{C}}$ is a Gr\"obner basis of the presentation ideal of the special fiber ring $\mathcal{F}\left(\bigoplus_{i=1}^{r} \ini_{\tau}(L)\right)$ with respect to some term order. It is straightforward to verify that $\mathcal{G}_{\mathcal{L} \times [r]} = \Lift(\mathcal{G}_{\mathcal{L}}) \cup \Quad_{\mathbf{C}}$.
\end{Remark}

\begin{Corollary}
    \label{CMFiber}
    The fiber ring $\calF(\bigoplus_{i=1}^{r} \ini_{\tau} (L))$ is a Koszul Cohen--Macaulay normal domain. Furthermore, $\calF(\bigoplus_{i=1}^{r} \ini_{\tau} (L))$ has rational singularities if the field $\KK$ has characteristic zero and is $F$-rational if $\KK$ has positive characteristic.
\end{Corollary}

\begin{proof}
    Since $\calF(\bigoplus_{i=1}^{r} \ini_{\tau} (L))$ is a Hibi ring, the first statement follows from \cite[Theorem  6.7, Corollary 6.21]{EHgbBook}.

    The argument from \cite[Corollary 2.3]{CHV96} supports the discussion on singularities. For completeness, we include it here. According to Hochster \cite[Proposition 1]{Hochster}, since $\calF(\bigoplus_{i=1}^{r} \ini_{\tau} (L))$ is normal, it is a direct summand of a polynomial ring. Consequently, by Boutot's theorem \cite{Boutot}, $\calF(\bigoplus_{i=1}^{r} \ini_{\tau} (L))$ possesses rational singularities if the characteristic of $\KK$ is zero. Furthermore, according to a theorem by Hochster and Huneke \cite{MR1044348}, $\calF(\bigoplus_{i=1}^{r} \ini_{\tau} (L))$ is strongly $F$-regular and particularly $F$-rational if $\KK$ has positive characteristic.
\end{proof}

The next  result shows that the initial ideal of $L$ has the weakly polymatroidal property.
This paves the way for determining the defining equations of the Rees algebra $\calR(\bigoplus_{i=1}^{r} \ini_{\tau} (L))$.
We first recall the definition of weakly polymatroidal ideal.

\begin{Definition}
    [{\cite[Definition 6.25]{EHgbBook}}]
    An equigenerated squarefree monomial ideal $J$ of the polynomial ring $
    \KK[x_1,\dots,x_n]$ is called \emph{weakly polymatroidal} with respect to the order $x_1>x_2>\cdots>x_n$ of variables, if for any two minimal monomial generators $f=x_{i_1}\cdots x_{i_d}$ with $i_1<\cdots<i_d$, and $g=x_{j_1}\cdots x_{j_d}$ with $j_1<\cdots<j_d$, such that $i_1=j_1, \dots,i_{t-1}=j_{t-1}$ and $i_t<j_t$, there exists $\ell\ge t$ such that $x_{i_t}(g/x_{j_{\ell}})\in J$.
\end{Definition}

\begin{Lemma}\label{weeklyPoly}
    The initial ideal of 
    $L=I_n(\bfX_{\bfL})$ is weakly polymatroidal with respect to the natural order of the variables in $\bfL$.
\end{Lemma}
\begin{proof}
    The ideal $\ini_{\tau}(L)$ is obviously squarefree and equigenerated. Suppose that $\bda,\bdb\in \calL$ satisfies $a_i=b_i$ for $i=1,2,\dots,t-1$ and $a_t<b_t$.  
    Let 
    \[
        \bdc\coloneqq \bdb+a_t\bde_t-b_t\bde_t=(b_1,\dots,b_{t-1},a_t,b_{t+1},\dots,b_n),
    \]
    where $\bde_t$ is the $t$-th unit vector in $\RR^n$.
    It is apparent that $\bdc\in\bdD$ and $\bdl^\dag\le \bdc\le \bdk^\dag$ for the $\bdl^\dag$ and $\bdk^\dag$ introduced in the proof of \Cref{lem:LexClosed}. Therefore $\bdc\in \calL$. This amounts to saying that $\ini_{\tau}(L)$ is weakly polymatroidal.
\end{proof}

\begin{Remark}
    \label{rmk:compatible_sorting}
    Let $\KK[T_{\bda}\mid \bda\in \calL]$ be the ambient ring for the presentation of the toric ring $\KK[\bfx_{\bda}\mid\bda\in\calL]$. 
    For each $\bda,\bdb\in \calL$, it is easy to see that $\sort(\bfx_{\bda},\bfx_{\bdb})=(\bfx_{\bda\wedge \bdb},\bfx_{\bda\vee\bdb})$ for the sorting operator due to Sturmfels. Therefore, the following orders coincide on $\KK[T_{\bda}\mid \bda\in\calL]$:
    \begin{itemize}
        \item compatible monomial orders as defined in Subsection \ref{subsection:Hibi};
        \item sorting orders in the sense of \cite[Theorem 6.14]{EHgbBook}.
    \end{itemize}
\end{Remark}

\begin{Theorem}
    \label{thm:InitialRees}
    Let $\sigma$ be any compatible order on $\KK[T_{\bda}\mid\bda\in \calL]$ and $\sigma'$ be an extended monomial order on $R[T_{\bda} \mid \bda \in \calL]$ where $R = \KK[\bfL]$.
    Let $\calG_{\calL}' \coloneqq \calG_{\calL} \cup H$ be the set obtained from $\calG_{\calL}$ by adding all the Eagon--Northcott type binomial relations. 
    Then, $\calG_{\calL}'$ is a Gr\"{o}bner basis of the presentation ideal of $\calR(\ini_{\tau} (L))$ with respect to
    $\sigma'$. 
\end{Theorem}

\begin{proof}
    According to \Cref{weeklyPoly} and \Cref{rmk:compatible_sorting}, as well as \cite[Proposition 6.26]{EHgbBook}, the ideal $\ini_{\tau} (L)$ satisfies the $\ell$-exchange property in the sense of \cite[Definition 6.23]{EHgbBook}.  
    Therefore, the Gr\"{o}bner basis result is an immediate consequence of \cite[Theorem 6.24]{EHgbBook}, and \Cref{FiberLadder}.
\end{proof}

Finally, we handle the multi-Rees algebra of the initial ideal.
Following 
\Cref{FiberLadder} and \Cref{thm:InitialRees}, let
\[
    \calG_{\calL \times [r]}' \coloneqq \Lift(\calG_{\calL}')  \subset  \KK[\bfZ,T_{\bda,j} \mid \bda\in \calL, j\in [r]]
\]
be the lifting of the Gr\"obner basis $\calG_{\calL}'$ defined in of the presentation ideal of $\calR(\ini_{\tau}(L))$ in $\KK[\bfL, T_{\bda} \mid \bda \in \calL]$.

\begin{Theorem}
    \label{thm:InitialMultiRees}
    With the above settings, the set 
    \(\calG_{\calL \times [r]}'\cup \Quad_{\bfC}\)
    is a Gr\"{o}bner basis of the presentation ideal of $\calR(\bigoplus_{i=1}^r\ini_{\tau} (L))$.
    In particular, $\calR( \bigoplus_{i=1}^r\ini_{\tau} (L))$ is a Koszul Cohen--Macaulay normal domain. Furthermore, $\calR( \bigoplus_{i=1}^r\ini_{\tau} (L))$ has rational singularities if the field $\KK$ has characteristic zero, and is $F$-rational if $\KK$ has positive characteristic.
\end{Theorem}

\begin{proof}
    The Gr\"{o}bner basis statement is the immediately consequence of Theorems \ref{GBMultiRessInitial_FiberProduct}, \ref{FiberLadder}, and \ref{thm:InitialRees}.  Since the presentation ideal of the toric ring $\calR(\bigoplus_{i=1}^{r} \ini_{\tau} (L))$ has a squarefree initial ideal, the ring $\calR(\bigoplus_{i=1}^{r} \ini_{\tau} (L))$ is normal by \cite[Proposition 13.15]{Sturmfels}. It is evident that $\calR(\bigoplus_{i=1}^{r} \ini_{\tau} (L))$ is isomorphic to the affine semigroup ring $ \KK\big[\bfL
    \cup \left\{\ini_{\tau}([\bda]_{\bfL})t_i\mid (\bda,i) \in \calL\times [r]\right\}\big]$. Consequently, as per \cite[Theorem 1]{Hochster}, $\calR(\bigoplus_{i=1}^{r} \ini_{\tau} (L))$ is Cohen--Macaulay.
    The remaining statements can be proved by a similar argument as that in the proof of \Cref{CMFiber}.
\end{proof}

\section{Multi-Rees algebras of maximal minors of ladder matrices}

The objective of this section is to compute a Gr\"obner basis for the presentation ideal of the multi-Rees algebra associated with the 
ideal $L=I_n(\bfX_{\bfL})$, which is generated by the maximal minors of the ladder matrix $\bfX_{\bfL}$. 
To achieve this, we will leverage the previously determined Gr\"obner basis for the presentation ideal of the multi-Rees algebra of the initial ideal, as described in Theorem \ref{thm:InitialRees}. 

In \Cref{LadderSet}, we identified a specific class of maximal minors that serve as minimal generators for \( L \). These minors will play a pivotal role in the subsequent Pl\"ucker relations, as detailed in \Cref{reesThirdEq} and \Cref{thm:sagbiLifts}.

\begin{Definition}
    \label{PluckersExist}
    Suppose that $(\bda,i_1)$ and $(\bdb,i_2)$ are two incomparable elements in the distributive product lattice $\calD\times [r]$. 
    \begin{enumerate}[a]
        \item \label{PluckersExist_a}
            Let $(\bdc,j_1) \coloneqq (\bda,i_1)\wedge(\bdb,i_2)$ and $(\bdd,j_2)\coloneqq (\bda,i_1)\vee(\bdb,i_2)$.

        \item \label{PluckersExist_b}
            For the indices $j_1$ and $j_2$ defined in part \ref{PluckersExist_a}, two tuples $$(\bde,j_1)=(e_1,e_2,\ldots,e_{n},j_1) \text{ and }(\bdg,j_2)=(g_1,g_2,\ldots,g_{n},j_2)$$ in $\calD \times [r]$ will be called \emph{vibrations of $(\bda,i_1)$ and $(\bdb,i_2)$}, 
            if $\bde\ne \bdc$, $\bdg \ne \bdd$, $\bde\le \bdg\wedge \bdc$, and we have an equality of multisets:
            \[
                \{e_1,e_2,\ldots,e_{n},g_1,g_2,\ldots,g_{n}\}=\{a_1,a_2,\ldots,a_{n},b_1,b_2,\ldots,b_{n}\}.
            \]
    \end{enumerate} 
\end{Definition}

A brief observation is necessary at this point.

\begin{Remark}
    Let $\calL$ denote the column index set associated with $\bfL$.
    Suppose that $(\bda,i_1)$ and $(\bdb,i_2)$ are two tuples in $\calL \times [r]$. Therefore, we have both $[\bda]_{\bfL}t_{i_1}$ and $[\bdb]_{\bfL}t_{i_2}$ in the Rees algebra $\calR(\bigoplus_{i=1}^{r}L)$. 
    If $(\bda,i_1)$ and $(\bdb,i_2)$ are incomparable, set $(\bda,i_1)\wedge (\bdb,i_2)=(\bdc,j_1)$ and $(\bda,i_1)\vee (\bdb,i_2)=(\bdd,j_2)$.
    It is an important fact,
        which we will demonstrate in the proof of \Cref{thm:sagbiLifts},
    that the product $[\boldsymbol{a}]_{\mathbf{L}} t_{i_1} \cdot [\boldsymbol{b}]_{\mathbf{L}} t_{i_2}$ in the Rees algebra $\mathcal{R}\left(\bigoplus_{i=1}^{r} L\right)$ can be expressed as a linear combination of products of the form $[\boldsymbol{c}]_{\mathbf{L}} t_{j_1} \cdot [\boldsymbol{d}]_{\mathbf{L}} t_{j_2}$, along with additional products of the form $[\boldsymbol{e}]_{\mathbf{L}} t_{j_1} \cdot [\boldsymbol{g}]_{\mathbf{L}} t_{j_2}$, where each pair $(\boldsymbol{e}, j_1)$ and $(\boldsymbol{g}, j_2)$ satisfies the assumptions outlined in \Cref{PluckersExist}.
    This property is well-established in the case where $r=1$ and $\calL=\calD$ under generic conditions.
    For a detailed exposition, we refer to the proof of Theorem 6.46 in \cite{EHgbBook}.
\end{Remark}

To determine the presentation ideal of $\mathcal{R}(\bigoplus_{i=1}^{r}L)$, we first describe its expected defining equations.

\begin{Definition}
    \label{reesThirdEq}
    Let $R[T_{\bda,i}\mid (\bda,i)\in \calL\times [r] ]$ be the ambient ring for the presentation ideal of $\calR(\bigoplus_{i=1}^r L)$, where $R=\KK[\bfL]$. 
    By abuse of notation, for $\bda\in \calD\setminus \calL$, we set $T_{\bda,i}=0$ in this ring. Similarly, for $x_{i,j}\in \KK[\bfX]\setminus R$, we set $x_{i,j}=0$ in this ambient ring.
    \begin{enumerate}[a]
        \item Let $\widetilde{\bdc} = \{c_1, \ldots, c_{n+1}\}$ be a set of integers such that $1\leq c_1<c_2<\cdots<c_{n+1}\le m$. For each $i\in [n]$ and $k\in [r]$, the polynomial of the form
            \begin{equation}
                \label{eqn:LinSyzLadder}
                \sum_{j\in[n+1]} (-1)^{i+j} x_{i,c_j} T_{\widetilde{\bdc}\setminus c_j,k},
            \end{equation}
            where $\widetilde{\bdc}\setminus c_j$ denotes the tuple $(c_1,\dots,c_{j-1},c_{j+1},\dots,c_{n+1})$ in $\calD$, will be called an \emph{Eagon--Northcott type relation}. This terminology arises from its connection to the Eagon--Northcott complex.

        \item \label{Plucker}Suppose that $(\bda,i_1)$ and $(\bdb,i_2)$ are two incomparable elements in $\calL\times [r]$.
            Let $(\bdc,j_1)=(\bda,i_1)\wedge (\bdb,i_2)$ and $(\bdd,j_2)=(\bda,i_1) \vee (\bdb,i_2)$, and let $(\bde,j_1)$ and $(\bdg,j_2)$ be vibrations of $(\bda,i_1)$ and  $(\bdb,i_2)$ as defined in \Cref{PluckersExist}. 
            Therefore, the sequence 
            \[
                (e_1,e_2,\ldots,e_n,g_1,g_2,\ldots,g_n) 
            \]
            is derived from the sequence 
            \[
                (a_1,a_2,\ldots,a_n,b_1,b_2,\ldots,b_n) 
            \]
            via a permutation $\iota_{\bde,\bdg}\in \frakS_{2n}$. In accordance with this notation, polynomials of the form
            \begin{equation}\label{eqn:PluckerLadder}
                T_{\bda,i_1}T_{\bdb,i_2}-T_{\bdc,j_1}T_{\bdd,j_2}+ \sum_{\substack{(\bde, j_1), (\bdg, j_2) \\ \text{vibrations of } (\bda, i_1), (\bdb, i_2)}} \sgn(\iota_{(\bde,j_1),(\bdg,j_2)}) T_{\bde,j_1}T_{\bdg,j_2}
            \end{equation}
            will be referred to as \emph{Pl\"{u}ker type relations}.
    \end{enumerate}
\end{Definition}

\begin{Remark}
    By \Cref{GWith2minors}, we have $\calG_{\calL\times [r]}=\Lift(\calG_{\calL})\cup \Quad_{\bfC}$. The set $\Quad_{\bfC}=\{ T_{\bda,i_1}T_{\bdb,i_2}-T_{\bda,i_2}T_{\bdb,i_1} \mid \bda< \bdb \in \calL \text{ and }i_1>i_2 \}$ arises from the incomparable pair $(\bda,i_1)$ and $(\bdb,i_2)$ in $\calL \times [r]$, where no vibrations exist. Consequently, \Cref{reesThirdEq} \ref{Plucker} inherently includes elements of $\Quad_{\bfC}$.
\end{Remark}

We are now prepared to state and prove the main theorem of this work.

\begin{Theorem}
    \label{thm:sagbiLifts}  
    Let $L=I_n(\bfX_{\bfL})$ be the ideal generated by the maximal minors of the ladder matrix $\bfX_{\bfL}$ over the polynomial ring $R=\KK[\bfL]$, and let $\calL$ denote the associated column index set. For any positive integer $r$, we have:

    \begin{enumerate}[a]
        \item \label{thm:sagbiLifts_a}  
            The polynomials of the set $\{[\bda]_{\bfL}t_i \mid (\bda,i) \in \calL\times [r]\}$ for a \textsc{Sagbi} basis of the $\mathbb{K}$-algebra they generate with respect to the monomial order $\tau'$ in \Cref{lexOrder}. In particular, 
            \[
                \ini_{\tau'}(\KK[[\bda]_{\bfL}t_i \mid (\bda,i) \in \calL\times [r]])= \KK[\ini_{\tau}([\bda]_{\bfL})t_i \mid (\bda,i) \in \calL\times [r]].
            \]
        \item \label{thm:sagbiLifts_b}  
            The polynomials in the set $ \bfL \cup \{[\bda]_{\bfL}t_i\mid (\bda,i) \in \calL\times [r]\}$ form a \textsc{Sagbi} basis of the Rees algebra $\mathcal{R}(\bigoplus_{i=1}^{r}L)$ with respect to the monomial order $\tau'$ as defined in \Cref{lexOrder}. 
            In particular, 
            \[
                \ini_{\tau'}\left(\calR \left(\bigoplus\nolimits_{i=1}^{r}L\right)\right)= \KK\big[
                    \bfL
                \cup \left\{\ini_{\tau}([\bda]_{\bfL})t_i\mid (\bda,i) \in \calL\times [r]\right\}\big]=\calR \left(\bigoplus\nolimits_{i=1}^{r} \ini_{\tau} (L)\right).
            \]

        \item \label{thm:sagbiLifts_c}  
            The presentation ideal of $\mathcal{F}\left(\bigoplus_{i=1}^{r} L\right)$ has a Gr\"obner basis consisting entirely of Pl\"ucker relations of the form \eqref{eqn:PluckerLadder}, with respect to a particular induced monomial order $\omega$ on $\KK[\mathbf{T}]\coloneqq\KK[T_{\bda,i}\mid (\bda,i)\in \calL\times [r] ]$.
        \item \label{thm:sagbiLifts_d}  
            With respect to some monomial ordering $\omega'$ on the ring $R[\bfT]\coloneqq R[T_{\bda,i}\mid (\bda,i)\in \calL\times [r]]$, the presentation ideal of $\mathcal{R}(\bigoplus_{i=1}^{r}L)$ possesses a Gr\"obner basis comprising Eagon--Northcott type relations in the form of \eqref{eqn:LinSyzLadder} and Pl\"ucker relations in the form of \eqref{eqn:PluckerLadder}.  In particular, $\calR(\bigoplus_{i=1}^{r}L)$ is of fiber type.
    \end{enumerate} 
\end{Theorem}
\begin{proof}
    In the following, we will focus on proving the \textsc{Sagbi} basis parts \ref{thm:sagbiLifts_a} and \ref{thm:sagbiLifts_b}. The proof of these statements is essentially a tailored version of \cite[Theorem 6.46]{EHgbBook} adapted to the multi-grading framework. The remaining statements \ref{thm:sagbiLifts_c} and \ref{thm:sagbiLifts_d} will follow from our proof for establishing the \textsc{Sagbi} basis, combined with \cite[Corollaries 2.1 and 2.2]{CHV96}.

    Recall that we introduced the map $\phi$ and $\phi^*$ in 
        Subsection \ref{subsection:blowup_algebras}.
    As a consequence of \Cref{thm:InitialRees}, the Eagon--Northcott type binomial relations and the nontrivial Hibi relations, as defined in \Cref{thm:InitialRees}, collectively form a Gr\"obner basis of $\ker(\phi^{\ast})$. Therefore, it suffices to show that any such binomial of the form $h=h_1-h_2$ can be ``lifted''. In other words, we only need to express $\phi(h)$ as a linear combination of elements of the form $\lambda \bdu \prod [\bda]_{\bfL}t_i$, where $\lambda \in \mathbb{K}\setminus \{0\}$, $\bdu$ is a monomial in the $x_{ij}$, $\bda \in \calL$, and $\ini_{\tau'}(\phi (h_1))=\ini_{\tau'}(\phi (h_2)) > \ini_{\tau'}(\bdu\prod [\bda]_{\bfL}t_i)$. For details, see \cite[Theorem 6.43]{EHgbBook}.
    The remaining verification follows from standard arguments.
    \begin{enumerate}[a]
        \item[\ref{thm:sagbiLifts_a}] The Eagon--Northcott type binomial relation in \eqref{eqn:ENT} can be lifted to the relation (\ref{eqn:LinSyzLadder}), as demonstrated by the following observations:
            \begin{align}
                (-1)^{i+i}x_{i, c_i}([\widetilde{\bdc}\setminus c_i]_{\bfL}t_k) & +(-1)^{i+i+1} x_{i, c_{i+1}}([\widetilde{\bdc}\setminus c_{i+1}]_{\bfL} t_k) \notag \\
                &\qquad = \sum_{j\in[n+1]\setminus \{i,i+1\}} (-1)^{i+j+1} x_{i,c_j} ([\widetilde{\bdc} \setminus c_j]_{\bfL} t_k), 
                \label{eqn:ENT_lift}
            \end{align}
            where $\widetilde{\bdc} = \{c_1, \ldots, c_{n+1}\}$ and $1\leq c_1<\cdots <c_{n+1}\leq m$. The equality holds because, from linear algebra, we have 
            \[
                \sum_{j\in[n+1]} (-1)^{i+j} x_{i,c_j} [\widetilde{\bdc} \setminus c_j]_{\bfL}=0.
            \]
            To show that \eqref{eqn:ENT_lift} gives the desired lifting,
            we analyze two cases bases on the position of $j$ relative to $i$:
            \begin{enumerate}[i]
                \item If $j<i$, the leading monomial of $x_{i,c_j} [\widetilde{\bdc}\setminus c_j]_{\bfL}$ is 
                    no bigger than
                    \[
                        x_{i,c_j}x_{1,c_1}\cdots x_{j-1, c_{j-1}}x_{j, c_{j+1}}\cdots x_{i ,c_{i+1}}\cdots x_{n, c_{n+1}}.
                    \]
                    On the other hand, the common leading monomial of $x_{i, c_i}[\widetilde{\bdc}\setminus c_i]_{\bfL}$ and $x_{i, c_{i+1}}[\widetilde{\bdc}\setminus c_{i+1}]_{\bfL}$ on the left of \eqref{eqn:ENT_lift} is 
                    \[
                        x_{i,c_i}x_{1,c_1}\cdots x_{j-1, c_{j-1}}x_{j, c_{j}}\cdots x_{i-1, c_{i-1}} x_{i, c_{i+1}}\cdots x_{n, c_{n+1}}. 
                    \]
                    Since $x_{j ,c_{j+1}} < x_{j, c_j}$, the initial monomials meet the desired condition.

                \item If instead $j>i+1$, then the leading monomial of  $x_{i,c_j} [\widetilde{\bdc}\setminus c_j]_{\bfL}$ is 
                    no bigger than
                    \[
                        x_{i,c_j}x_{1,c_1}\cdots x_{i-1,c_{i-1}}x_{i,c_{i}} x_{i+1,c_{i+1}}\cdots x_{j-1 c_{j-1}}x_{j c_{j+1}}\cdots x_{m c_{m+1}}.
                    \]
                    The common leading monomial of $x_{i, c_i}[\widetilde{\bdc}\setminus c_i]_{\bfL}$ and $x_{i, c_{i+1}}[\widetilde{\bdc}\setminus c_{i+1}]_{\bfL}$ on the left-hand side is
                    \[
                        x_{i,c_i}x_{1,c_1}\cdots x_{i-1,c_{i-1}} x_{i,c_{i+1}}\cdots x_{j-1,c_{j}}x_{j,c_{j+1}}\cdots x_{m,c_{m+1}}. 
                    \]
                    Since $x_{i,c_j} < x_{i,c_{i+1}}$, the initial monomials meet the desired condition.
            \end{enumerate} 
        \item[\ref{thm:sagbiLifts_b}] Suppose that $(\bda,i_1)$ and $(\bdb,i_2)$ are two incomparable elements in $\calL\times [r]$. Let $(\bdc,j_1)=(\bda,i_1) \wedge (\bdb,i_2)$ and $(\bdd,j_2)=(\bda,i_1)\vee (\bdb,i_2)$. We therefore analyze the lifting of the nontrivial Hibi relation $T_{\bda,i_1}T_{\bdb,i_2}-T_{\bdc,j_1}T_{\bdd,j_2}$ to the Pl\"ucker type relation in (\ref{eqn:PluckerLadder}). The analysis depends on the relationship between the elements $(\bda,i_1)$ and $(\bdb,i_2)$.
            \begin{enumerate}[i]
                \item When $\bda<\bdb$ but $i_1>i_2$, we have 
                    $(\bdc, j_1) = (\bda, i_2)$ and $(\bdd, j_2) = (\bdb, i_1)$. Hence
                    \[
                        [\bda]_{\bfL}\, t_{i_1} \cdot [\bdb]_{\bfL}\, t_{i_2} - [\bdc]_{\bfL}\, t_{j_1} \cdot [\bdd]_{\bfL}\, t_{j_2}
                        = 0.
                    \]
                    No further action is needed in this case.

                \item When $\bdc \notin \{\bda, \bdb\}$ and $\bdd \notin \{\bda, \bdb\}$, we claim that
                    \begin{equation}
                        [\bda]_{\bfL} t_{i_1} \cdot [\bdb]_{\bfL}t_{i_2} - [\bdc]_{\bfL}t_{j_1} \cdot [\bdd]_{\bfL} t_{j_2}  = -\sum \sgn(\iota_{(\bde,j_1),(\bdg,j_2)}) [\bde]_{\bfL}t_{j_1}\cdot [\bdg]_{\bfL}t_{j_2}. 
                        \label{eqn:Plucker_lift}
                    \end{equation}
                    In this expression, the sum is taken over vibrations $(\bde,j_1)$ and $(\bdg,j_2)$ of $(\bda,i_1)$ and $(\bdb,i_2)$ as defined in \Cref{PluckersExist}.  Whence, the sequence $(e_1,e_2,\ldots,e_n,g_1,g_2,\ldots,g_n)$ arises from the sequence $(a_1,a_2,\ldots,,a_n,b_1,b_2,\ldots,b_n)$ by applying a permutation $\iota_{\bde,\bdg}\in \frakS_{2n}$. 

                    To substantiate the assertion, for the $(\bda,i_1)=(a_1,\dots,a_n,i_1)\in \calD\times [r]$, let $[\bda]$ be the maximal minor of $\bfX$ spanned by the columns corresponding to the subtuple $\bda=(a_1,\dots,a_n)$. Similarly, the elements $[\bdb],[\bdc]$, and so on are defined.
                    The equality in \eqref{eqn:Plucker_lift} is guaranteed by the simplified form:
                    \[
                        [\bda][\bdb]-[\bdc][\bdd]=-\sum \sgn(\iota_{\bde,\bdg})[\bde][\bdg]
                    \]
                    from \cite[Lemma 7.2.3]{brunsherzog}, under the obvious specialization from $\bfX$ to $\bfX_{\bfL}$. Moreover, it is well-established that $\ini_{\tau} ([\bde][\bdg]) < \ini_{\tau}([\bda][\bdb])=\ini_{\tau}([\bdc][\bdd])$. This can be seen, for instance, 
                    from the proof of \cite[Theorem 6.46]{EHgbBook}. 
                    Consequently, the initial monomials satisfy the requisite condition.
                    \qedhere
            \end{enumerate} 
    \end{enumerate} 
\end{proof}

\begin{Example}
    Let us consider the ideal $L=I_3(\bfX_{\bfL})$, where $\bfX_{\bfL}$ represents the ladder matrix in \Cref{LadderMatrix}. By lifting the defining equations of  $\calR(\bigoplus_{i=1}^{3} \ini _{\tau}(L))$, we obtain the defining equations of $\mathcal{R}(\bigoplus_{i=1}^{3} L)$.  As illustrated in the proof of \Cref{thm:sagbiLifts}, there are two distinct types of relations in the presentation ideal of $\mathcal{R}(\bigoplus_{i=1}^{3} L)$.
    \begin{enumerate}
        \item The first type of relations are the Eagon--Northcott relations. Examples of these include 
            \begin{gather*} 
                x_{1,1}T_{234,1}-x_{1,2}T_{134,1},  
                \intertext{and}
                x_{2,4}T_{356,3}-x_{2,5}T_{346,3}-x_{2,3}T_{456,3}+x_{2,6}T_{345,3}.
            \end{gather*}
        \item The second type of relations are the Pl\"ucker relations, exemplified by
            \begin{gather*} 
                T_{134,3}T_{245,2}-T_{134,2}T_{245,3},\\
                T_{145,1}T_{236,3}-T_{135,1}T_{246,3} +T_{134,1}T_{256,3},
                \intertext{and}   
                T_{156,2}T_{348,3}-T_{146,2}T_{358,3} +T_{145,2}T_{368,3}+T_{136,2}T_{458,3}-T_{135,2}T_{468,3}+T_{134,2}T_{568,3}.
            \end{gather*}
    \end{enumerate}
    Notice that the maximal minors encompassing the initial two columns of $\bfX_\bfL$ are all equal to zero.
\end{Example}

\begin{Corollary}
    \label{cor:propertiesReesLadder} 
    Let $L=I_n(\bfX_{\bfL})$ be the ideal generated by the maximal minors of the ladder matrix $\bfX_{\bfL}$, and $r$ be a positive integer. Then the blowup algebra $\mathcal{R}(\bigoplus_{i=1}^{r}L)$ and $\mathcal{F}(\bigoplus_{i=1}^{r}L)$ are Koszul Cohen--Macaulay normal domains. In particular, these two blowup algebras have rational singularities if $\Char\mathbb{K} = 0$, and they are F-rational if $\Char(\mathbb{K})>0$. Furthermore, the powers of $L$ have a linear resolution.
\end{Corollary}

\begin{proof}
    The first two conclusions follow directly from the application of \cite[Corollary 2.3]{CHV96} in conjunction with Theorems \ref{thm:InitialMultiRees} and \ref{thm:sagbiLifts}. 
    The final one follows from \cite[Corollary 3.6]{Blum}, since $\mathcal{R}(L)$ is Koszul.
\end{proof}

\section{Multi-fiber of unit interval ideals}
The blowup algebras of the unit interval ideal have been previously explored in \cite{ALL1}. In this final section, we aim to demonstrate that its analysis
aligns seamlessly with the framework established in our current investigation.
As previously stated, let $\bfX = (x_{i,j})$ be an $n\times m$ matrix of indeterminates where $n\leq m$. Furthermore, let $R = \mathbb{K}[\bfX]$ be the polynomial ring over a field $\mathbb{K}$ in the indeterminates $x_{i,j}$, with the term order $\tau$ specified in \Cref{lexOrder}.

\begin{Setting}
    \label{def:unit} 
    \begin{enumerate}[a]
        \item In the case where $1\le u<v\le m$, 
            the ideal $I_n(\bfX_{[u,v]})$ is defined as the ideal of $R$ generated by the $n$-minors of the submatrix 
            \[
                \begin{bmatrix}
                    x_{1,u} & x_{1,u+1} & \cdots & x_{1,v}\\
                    x_{2,u} & x_{2,u+1} & \cdots & x_{2,v}\\
                    \vdots & \vdots & & \vdots \\
                    x_{n,u} & x_{n,u+1} & \cdots & x_{n,v}
                \end{bmatrix}.
            \]

        \item In the following, we will assume that $[1,m]$ is partitioned as $\bigcup_{i=1}^s \calI_i$, where $\calI_i=[u_i, v_i]$ for $i\in [s]$. Without loss of generality, we may suppose that this representation is irredundant, meaning no interval $\calI_i$ is contained within another interval $\calI_j$ for $i\ne j$. Consequently, it can be assumed that 
            \begin{align*}
                1=u_1< u_2 < \cdots < u_s<m \qquad
                \text{and} \qquad
                1<v_1<v_2<\cdots <v_s=m.
            \end{align*}

        \item The 
            ideal $U=\sum_{i=1}^s I_n(\bfX_{\calI_i})$ in $R$ will be referred to as a \emph{unit interval determinantal ideal}. Additionally, the \emph{associated column index set} is defined as:
            \[
                \calU\coloneqq \{(a_1,\ldots,a_n)\in \calD \mid \text{$a_1,a_n\in \calI_i$ simultaneously for some $i\in [s]$} \}.
            \]
            If we continue to let $[\bda]$ denote
            the maximal minor of $\bfX$ spanned by the columns corresponding to the tuple $\bda=(a_1,\dots,a_n)\in \calD$, then 
            $U$ is generated by all $[\bda]$ with $\bda\in \calU$.
    \end{enumerate} 
\end{Setting}

\begin{Lemma}
    [{\cite[Lemmas 3.4 and 4.1]{ALL1}}]
    \label{minMaxExist}
    \label{UnitPluckersExist}
    \begin{enumerate}[a]
        \item The subset $\calU$ is closed under the lattice operations of $\calD$, making it a sublattice of $\calD$. In other words, if $\bda$ and $\bdb$ are two incomparable elements in $\calU$, then both the meet $\bdc\coloneqq \bda\wedge \bdb$ and the join $\bdd\coloneqq \bda\vee \bdb$ in $\calD$ actually belong to $\calU$.
        \item Furthermore, let  $\bde$ and $\bdg$ be the vibrations of the incomparable elements $\bda$ and $\bdb$ in $\calU$, as previously defined in \Cref{PluckersExist}. Then, both $\bde$ and $\bdg$ are elements of 
            $\calU$. 
    \end{enumerate}
\end{Lemma}

Building upon the same line of reasoning presented in \Cref{FiberLadder} and \Cref{CMFiber}, we arrive at the following theorem by applying \Cref{minMaxExist}.

\begin{Theorem}
    \label{FiberUnit}
    Let $U$ be a unit interval determinantal ideal, and $\ini_{\tau} (U)$ be its initial ideal. Then the set $\calG_{\calU\times [r]}$ of nontrivial Hibi relations constitutes a Gr\"{o}bner basis of the presentation ideal of the fiber ring $\calF(\bigoplus_{i=1}^{r} \ini_{\tau} (U))$. In particular, $\calF(\bigoplus_{i=1}^{r} \ini_{\tau} (U))$ is a Koszul Cohen--Macaulay normal domain. Furthermore,  if the field $\KK$ has characteristic zero, then $\calF(\bigoplus_{i=1}^{r} \ini_{\tau} (U))$ has rational singularities; if $\KK$ has positive characteristic, then $\calF(\bigoplus_{i=1}^{r} \ini_{\tau} (U))$ is $F$-rational.
\end{Theorem}

Notice that the description of the defining equations of $\calF(\bigoplus_{i=1}^{r} \ini_{\tau} (U))$ bears resemblance to the ladder matrix case. However, it is actually constituted by different collections of binomials.

Next, we describe the defining equations of $\mathcal{F}(\bigoplus_{i=1}^{r}U)$. 
Notice that if $\bda,\bdb$ are incomparable elements in $\calU$, then the induced relationships of Pl\"{u}ker type (in the form of \eqref{eqn:PluckerLadder}) belong to the polynomial ring $\KK[T_{\bdu}\mid \bdu\in \calU]$ by \Cref{UnitPluckersExist}.

\begin{Theorem}
    \label{thm:sagbiLiftsUnit}  
    Let $U$ be a unit interval determinantal ideal over the polynomial ring $R=\KK[\bfX]$, and let $\calU$ denote the associated column index set. The polynomials of the set $\{[\bda]t_i \mid (\bda,i)\in \calU\times [r]\} $ form a \textsc{Sagbi} basis of the $\KK$-algebra that they generate, with respect to the monomial order $\tau'$ specified in \Cref{lexOrder}. In particular, 
    \[
        \ini_{\tau'}(\KK[[\bda]t_i \mid (\bda,i) \in \calU\times [r]])= \KK[\ini_{\tau}([\bda])t_i \mid (\bda,i) \in \calU\times [r]].
    \]
    Moreover, the presentation ideal of the fiber ring $\mathcal{F}(\bigoplus_{i=1}^{r}U)$ has a Gr\"obner basis given by the induced Pl\"ucker relations in the forms of \eqref{eqn:PluckerLadder} with respect to some induced monomial order $\omega$ on $\KK[T_{\bdu}\mid\bdu\in\calU]$. Furthermore, $\mathcal{F}(\bigoplus_{i=1}^{r}U)$ is a Koszul normal Cohen--Macaulay domain. Finally, if the field $\KK$ has characteristic zero, then $\calF(\bigoplus_{i=1}^{r} U)$ has rational singularities; if $\KK$ has positive characteristic, then $\calF(\bigoplus_{i=1}^{r} U)$ is $F$-rational.
\end{Theorem}

In the light of \Cref{minMaxExist}, the proof of \Cref{thm:sagbiLiftsUnit} can be derived from the same line of reasoning presented in \Cref{CMFiber} and \Cref{thm:sagbiLifts}, which will be omitted here for brevity.

We conclude this work with the multi-Rees algebras of unit interval determinantal ideals.

\begin{Definition}\label{reesEQ}
    Let $U$ be a unit interval determinantal ideal, and let $\ini_{\tau} (U)$ be its initial ideal. 
    \begin{enumerate}[a]
        \item Suppose that $\bda>\bdb$ are two elements in $\calU$. Let $i_{\bda}\coloneqq \min \{i\mid \bda\in \calI_i\}$ and similarly define $i_{\bdb}$. Assume in addition that $i_{\bda}<i_{\bdb}$ and $\bda \in \calI_{i_a}\backslash \calI_{i_b}$. In this case, we have the \emph{Koszul-type binomial}:
            \begin{equation*}
                \bdx_\bda T_{\bdb}-\bdx_\bdb T_{\bda}. 
            \end{equation*}
            Let $K$ be the collection of all such Koszul-type binomials.  Furthermore, for $k\in [r]$, we have the \emph{Koszul-type polynomial}:
            \begin{equation}\label{eq: KoszulClosedDFI}
                [\bda] T_{\bdb,k}-[\bdb] T_{\bda,k}. 
            \end{equation}

        \item Let $\widetilde{\bdc} = \{c_1, \ldots, c_{n+1}\} \subseteq \calI_i$  for some $i$ where $1\leq c_1<\cdots <c_{n+1}\leq m$.  
           Then by Eagon--Northcott complex, we have an \emph{Eagon--Northcott-type} binomial relation (related to $\calU$):
            \begin{equation*}
                {x_{i c_i}T_{\widetilde{\bdc}\setminus c_i}}-x_{i c_{i+1}}T_{\widetilde{\bdc}\setminus c_{i+1}}.
            \end{equation*}
            Let $H$ be the collection of all such Eagon--Northcott-type binomials.
            Furthermore, for $k\in [r]$, we have the \emph{Eagon-Northcott-type polynomials}:
            \begin{equation}\label{eq: LinSyzClosedDFI}
                \sum_{j\in[n+1]} (-1)^{i+j} x_{i,c_j} T_{\widetilde{\bdc}\setminus c_j,k}.
            \end{equation}

        \item The set $\calG_{\calU}$ of nontrivial Hibi relations is a Gr\"{o}bner basis of the presentation ideal of $\calF(\ini_{\tau} (U))$ by \Cref{FiberUnit}. Let $\calG_{\calU}'\coloneqq\calG_{\calU}\cup K \cup H$ be the set obtained from $\calG_{\calU}$ by supplementing all the {Koszul-type} binomial relations and the {Eagon--Northcott type} binomial relations. By \cite[Theorem 3.14]{ALL1}, 
            $\calG_{\calU}'$ is a Gr\"{o}bner basis of the presentation ideal of $\calR(\ini_{\tau} (U))$ in $\KK[\bfX,T_{\bda} \mid \bda \in \calU]$. 
            
            To study the multi-Rees algebra, we write 
            $\calG_{\calU \times [r]}'\coloneqq\Lift (\calG_{\calU}')$
            where the operation $\Lift$ is defined in \Cref{EQForMulti-ReesInitial}. Furthermore, let 
            $\bfC=(T_{\bda,j})_{\bda\in \calU, j\in [r]}$ be a $|\calU| \times r$ matrix of indeterminates. Denote the set of $2$-minors of $\bfC$ by $\Quad_{\bfC}$.
    \end{enumerate}
\end{Definition}

Using \cite[Theorem 3.14]{ALL1} and \Cref{fibertoRees} along with the same line of reasoning presented in Theorems \ref{thm:InitialRees} and \ref{thm:sagbiLifts}, we have the following result: 
\begin{Theorem}
    \label{ReesUnit}
    Let $U$ be a unit interval determinantal ideal, and $\ini_{\tau} (U)$ be its initial ideal. 
    \begin{enumerate}[a]
        \item  The set $\calG_{\calU\times [r]}'\cup \Quad_{\bfC}$ is a Gr\"{o}bner basis of the presentation ideal of the multi-Rees algebra $\calR(\bigoplus_{i=1}^{r} \ini_{\tau} (U))$. In particular, $\calR(\bigoplus_{i=1}^{r} \ini_{\tau} (U))$ is a Cohen--Macaulay normal domain. Furthermore,  if the field $\KK$ has characteristic zero, then $\calR(\bigoplus_{i=1}^{r} \ini_{\tau} (U))$ has rational singularities; if $\KK$ has positive characteristic, then $\calR(\bigoplus_{i=1}^{r} \ini_{\tau} (U))$ is $F$-rational.
        \item    The polynomials of the set $\bfX \cup \{[\bda]t_i \mid (\bda,i)\in \calU\times [r]\} $ form a \textsc{Sagbi} basis of the $\KK$-algebra that they generate, with respect to the monomial order $\tau'$ specified in \Cref{lexOrder}. In particular, 
            \[
                \ini_{\tau'}(\KK[\bfX \cup \{[\bda]t_i \mid (\bda,i) \in \calU\times [r]\}])= \KK[\bfX \cup \{\ini_{\tau}([\bda])t_i \mid (\bda,i) \in \calU\times [r]\}.
            \]
            Moreover, the presentation ideal of $\mathcal{R}(\bigoplus_{i=1}^{r}U)$ has a Gr\"obner basis given by the induced Pl\"ucker relations in the form of \eqref{eqn:PluckerLadder}, Koszul-type polynomials in the form of \eqref{eq: KoszulClosedDFI}, and Eagon--Northcott type polynomials in the form of \eqref{eq: LinSyzClosedDFI}, with respect to some induced monomial order $\omega$ on $\KK[T_{\bdu}\mid\bdu\in\calU]$.
            Furthermore, $\mathcal{R}(\bigoplus_{i=1}^{r}U)$ is a normal Cohen--Macaulay domain. Finally, if the field $\KK$ has characteristic zero, then $\calR(\bigoplus_{i=1}^{r} U)$ has rational singularities; if $\KK$ has positive characteristic, then $\calR(\bigoplus_{i=1}^{r} U)$ is $F$-rational.
    \end{enumerate} 
\end{Theorem}

\begin{acknowledgment*}
    The second author is partially supported by the ``Innovation Program for Quantum Science and Technology'' (2021ZD0302902) and ``the Fundamental Research Funds for the Central Universities''.
\end{acknowledgment*}

\noindent\textbf{Declaration of generative AI and AI-assisted technologies in the writing process:}
During the preparation of this work the authors used ChatGPT in order to improve readability and language of the work. After using this tool, the authors reviewed and edited the content as needed and take full responsibility for the content of the publication.

\bibliography{ReesBib}

\begin{bibdiv}
\begin{biblist}

\bib{MR926272}{book}{
      author={Abhyankar, Shreeram~S.},
       title={Enumerative combinatorics of {Y}oung tableaux},
      series={Monographs and Textbooks in Pure and Applied Mathematics},
   publisher={Marcel Dekker, Inc., New York},
        date={1988},
      volume={115},
        ISBN={0-8247-7816-2},
      review={\MR{926272}},
}

\bib{ALL1}{article}{
      author={Almousa, Ayah},
      author={Lin, Kuei-Nuan},
      author={Liske, Whitney},
       title={Rees algebras of unit interval determinantal facet ideals},
        date={2024},
        ISSN={0022-4049,1873-1376},
     journal={J. Pure Appl. Algebra},
      volume={228},
       pages={Paper No. 107601, 15},
         url={https://doi.org/10.1016/j.jpaa.2023.107601},
      review={\MR{4688113}},
}

\bib{MR789425}{article}{
      author={Backelin, J\"orgen},
      author={Fr\"oberg, Ralf},
       title={Koszul algebras, {V}eronese subrings and rings with linear resolutions},
        date={1985},
        ISSN={0035-3965},
     journal={Rev. Roumaine Math. Pures Appl.},
      volume={30},
       pages={85\ndash 97},
      review={\MR{789425}},
}

\bib{Blum}{article}{
      author={Blum, Stefan},
       title={Subalgebras of bigraded {K}oszul algebras},
        date={2001},
        ISSN={0021-8693,1090-266X},
     journal={J. Algebra},
      volume={242},
       pages={795\ndash 809},
         url={https://doi.org/10.1006/jabr.2001.8804},
      review={\MR{1848973}},
}

\bib{Boocher}{article}{
      author={Boocher, Adam},
       title={Free resolutions and sparse determinantal ideals,},
        date={2012},
     journal={Math. Res. Lett.},
      volume={19},
       pages={805\ndash 821},
}

\bib{Boutot}{article}{
      author={Boutot, Jean-Fran\c{c}ois},
       title={Singularit\'{e}s rationnelles et quotients par les groupes r\'{e}ductifs},
        date={1987},
        ISSN={0020-9910,1432-1297},
     journal={Invent. Math.},
      volume={88},
       pages={65\ndash 68},
         url={https://doi.org/10.1007/BF01405091},
      review={\MR{877006}},
}

\bib{BCresPowers}{incollection}{
      author={Bruns, Winfried},
      author={Conca, Aldo},
       title={Linear resolutions of powers and products},
        date={2017},
   booktitle={Singularities and computer algebra},
   publisher={Springer, Cham},
       pages={47\ndash 69},
      review={\MR{3675721}},
}

\bib{BCRV}{book}{
      author={Bruns, Winfried},
      author={Conca, Aldo},
      author={Raicu, Claudiu},
      author={Varbaro, Matteo},
       title={Determinants, {G}röbner bases, and cohomology},
      series={Springer Monographs in Mathematics},
   publisher={Springer Cham},
        date={1988},
}

\bib{brunsherzog}{book}{
      author={Bruns, Winfried},
      author={Herzog, J{\"u}rgen},
       title={Cohen-{Macaulay} rings.},
    language={English},
     edition={Rev. ed.},
      series={Camb. Stud. Adv. Math.},
   publisher={Cambridge: Cambridge University Press},
        date={1998},
      volume={39},
        ISBN={0-521-56674-6},
}

\bib{CDFGLPS}{article}{
      author={Celikbas, Ela},
      author={Dufresne, Emilie},
      author={Fouli, Louiza},
      author={Gorla, Elisa},
      author={Lin, Kuei-Nuan},
      author={Polini, Claudia},
      author={Swanson, Irena},
       title={Rees algebras of sparse determinantal ideals},
        date={2024},
        ISSN={0002-9947,1088-6850},
     journal={Trans. Amer. Math. Soc.},
      volume={377},
       pages={2317\ndash 2333},
         url={https://doi.org/10.1090/tran/9101},
      review={\MR{4744759}},
}

\bib{CWL}{article}{
      author={Chen, F.},
      author={Wang, W.},
      author={Liu, Y.},
       title={Computing singular points of plane rational curves},
        date={2008},
     journal={J. Symbolic Comput.},
      volume={43},
       pages={92\ndash 117},
}

\bib{MR1319965}{article}{
      author={Conca, Aldo},
       title={Ladder determinantal rings},
        date={1995},
        ISSN={0022-4049,1873-1376},
     journal={J. Pure Appl. Algebra},
      volume={98},
       pages={119\ndash 134},
         url={https://doi.org/10.1016/0022-4049(94)00039-L},
      review={\MR{1319965}},
}

\bib{CD}{article}{
      author={Conca, Aldo},
      author={De~Negri, Emanuela},
       title={{$M$}-sequences, graph ideals, and ladder ideals of linear type},
        date={1999},
        ISSN={0021-8693,1090-266X},
     journal={J. Algebra},
      volume={211},
       pages={599\ndash 624},
         url={https://doi.org/10.1006/jabr.1998.7740},
      review={\MR{1666661}},
}

\bib{CHV96}{article}{
      author={Conca, Aldo},
      author={Herzog, Jurgen},
      author={Valla, Giuseppe},
       title={Sagbi bases with applications to blow-up algebras},
        date={1996},
     journal={Journal fur die Reine und Angewandte Mathematik},
      volume={474},
       pages={113\ndash 138},
}

\bib{wicaII}{article}{
      author={Costantini, Alessandra},
      author={Fouli, Louiza},
      author={Goel, Kriti},
      author={Lin, Kuei-Nuan},
      author={Lindo, Haydee},
      author={Liske, Whitney},
      author={Mostafazadehfard, Maral},
       title={Algebraic invariants of the special fiber ring of ladder determinantal modules},
        date={2025},
        note={preprint},
}

\bib{CLS}{article}{
      author={Cox, David},
      author={Lin, Kuei-Nuan},
      author={Sosa, Gabriel},
       title={Multi-{R}ees algebras and toric dynamical systems},
        date={2019},
     journal={Proc. Amer. Math. Soc.},
      volume={147},
       pages={4605\ndash 4616},
}

\bib{Cox}{book}{
      author={Cox, David~A.},
       title={Applications of polynomial systems},
      series={CBMS Regional Conference Series in Mathematics},
   publisher={American Mathematical Society},
        date={2020},
      volume={134},
        ISBN={978-1-4704-5137-0},
}

\bib{DeNegri}{article}{
      author={De~Negri, Emanuela},
       title={Toric rings generated by special stable sets of monomials},
        date={1999},
     journal={Math. Nachr.},
      volume={203},
       pages={31\ndash 45},
      review={\MR{1698638}},
}

\bib{EHU}{article}{
      author={Eisenbud, David},
      author={Huneke, Craig},
      author={Ulrich, Bernd},
       title={What is the {R}ees algebra of a module?},
        date={2003},
        ISSN={0002-9939,1088-6826},
     journal={Proc. Amer. Math. Soc.},
      volume={131},
       pages={701\ndash 708},
         url={https://doi.org/10.1090/S0002-9939-02-06575-9},
      review={\MR{1937406}},
}

\bib{EHgbBook}{book}{
      author={Ene, Viviana},
      author={Herzog, J\"{u}rgen},
       title={Gr\"{o}bner bases in commutative algebra},
      series={Graduate Studies in Mathematics},
   publisher={American Mathematical Society, Providence, RI},
        date={2012},
      volume={130},
        ISBN={978-0-8218-7287-1},
         url={https://doi.org/10.1090/gsm/130},
      review={\MR{2850142}},
}

\bib{Yang_Baxter_algebras}{article}{
      author={Gateva-Ivanova, Tatiana},
       title={Segre products and {S}egre morphisms in a class of {Y}ang-{B}axter algebras},
        date={2023},
        ISSN={0377-9017,1573-0530},
     journal={Lett. Math. Phys.},
      volume={113},
       pages={Paper No. 34, 34},
         url={https://doi.org/10.1007/s11005-023-01657-z},
      review={\MR{4562210}},
}

\bib{GWGraded}{article}{
      author={Goto, Shiro},
      author={Watanabe, Keiichi},
       title={On graded rings. {I}},
        date={1978},
        ISSN={0025-5645},
     journal={J. Math. Soc. Japan},
      volume={30},
       pages={179\ndash 213},
         url={https://doi.org/10.2969/jmsj/03020179},
      review={\MR{494707}},
}

\bib{MR3838370}{book}{
      author={Herzog, J\"urgen},
      author={Hibi, Takayuki},
      author={Ohsugi, Hidefumi},
       title={Binomial ideals},
      series={Graduate Texts in Mathematics},
   publisher={Springer, Cham},
        date={2018},
      volume={279},
        ISBN={978-3-319-95347-2; 978-3-319-95349-6},
         url={https://doi.org/10.1007/978-3-319-95349-6},
      review={\MR{3838370}},
}

\bib{Hochster}{article}{
      author={Hochster, M.},
       title={Rings of invariants of tori, {C}ohen-{M}acaulay rings generated by monomials, and polytopes},
        date={1972},
        ISSN={0003-486X},
     journal={Ann. of Math.},
      volume={96},
       pages={318\ndash 337},
         url={https://doi.org/10.2307/1970791},
      review={\MR{304376}},
}

\bib{MR1044348}{incollection}{
      author={Hochster, Melvin},
      author={Huneke, Craig},
       title={Tight closure and strong {$F$}-regularity},
        date={1989},
       pages={119\ndash 133},
        note={Colloque en l'honneur de Pierre Samuel (Orsay, 1987)},
      review={\MR{1044348}},
}

\bib{Likelihood}{article}{
      author={Kahle, Thomas},
      author={K\"{u}hne, Lukas},
      author={M\"{u}hlherr, Leonie},
      author={Sturmfels, Bernd},
      author={Wiesmann, Maximilian},
       title={Arrangements and likelihood},
        date={2024},
     journal={Vietnam J. Math.},
      eprint={arXiv:2411.09508},
}

\bib{KRSegre}{article}{
      author={Kahle, Thomas},
      author={Rauh, Johannes},
       title={Toric fiber products versus {S}egre products},
        date={2014},
        ISSN={0025-5858,1865-8784},
     journal={Abh. Math. Semin. Univ. Hambg.},
      volume={84},
       pages={187\ndash 201},
         url={https://doi.org/10.1007/s12188-014-0095-5},
      review={\MR{3267741}},
}

\bib{MR2761126}{article}{
      author={Lakshmibai, V.},
      author={Mukherjee, H.},
       title={Singular loci of {H}ibi toric varieties},
        date={2011},
        ISSN={0970-1249,2320-3110},
     journal={J. Ramanujan Math. Soc.},
      volume={26},
      number={1},
       pages={1\ndash 29},
      review={\MR{2761126}},
}

\bib{LinReesModule}{article}{
      author={Lin, Kuei-Nuan},
       title={Cohen-{M}acaulayness of {R}ees algebras of modules},
        date={2016},
        ISSN={0092-7872,1532-4125},
     journal={Comm. Algebra},
      volume={44},
       pages={3673\ndash 3682},
         url={https://doi.org/10.1080/00927872.2015.1086925},
      review={\MR{3503377}},
}

\bib{LSCMNormal}{article}{
      author={Lin, Kuei-Nuan},
      author={Shen, Yi-Huang},
       title={Fiber cones of rational normal scrolls are {C}ohen-{M}acaulay},
        date={2022},
     journal={J. Algebraic Combin.},
      volume={56},
       pages={547\ndash 563},
}

\bib{MR991408}{article}{
      author={Mulay, S.~B.},
       title={Determinantal loci and the flag variety},
        date={1989},
        ISSN={0001-8708,1090-2082},
     journal={Adv. Math.},
      volume={74},
       pages={1\ndash 30},
         url={https://doi.org/10.1016/0001-8708(89)90002-9},
      review={\MR{991408}},
}

\bib{RSMatrix}{article}{
      author={Ramos, Zaqueu},
      author={Simis, Aron},
       title={Matrices over polynomial rings approached by commutative algebra},
        date={2024},
      eprint={arXiv:2406.04266},
}

\bib{Shibuta}{article}{
      author={Shibuta, Takafumi},
       title={Gr\"obner bases of contraction ideals},
        date={2012},
        ISSN={0925-9899,1572-9192},
     journal={J. Algebraic Combin.},
      volume={36},
       pages={1\ndash 19},
         url={https://doi.org/10.1007/s10801-011-0320-6},
      review={\MR{2927653}},
}

\bib{SUVModule}{article}{
      author={Simis, Aron},
      author={Ulrich, Bernd},
      author={Vasconcelos, Wolmer~V.},
       title={Rees algebras of modules},
        date={2003},
        ISSN={0024-6115,1460-244X},
     journal={Proc. London Math. Soc.},
      volume={87},
       pages={610\ndash 646},
         url={https://doi.org/10.1112/S0024611502014144},
      review={\MR{2005877}},
}

\bib{Sturmfels}{book}{
      author={Sturmfels, Bernd},
       title={Gr\"{o}bner bases and convex polytopes},
      series={University Lecture Series},
   publisher={American Mathematical Society, Providence, RI},
        date={1996},
      volume={8},
        ISBN={0-8218-0487-1},
         url={https://doi.org/10.1090/ulect/008},
      review={\MR{1363949}},
}

\bib{SullivantToricFiber}{article}{
      author={Sullivant, Seth},
       title={Toric fiber products},
        date={2007},
        ISSN={0021-8693,1090-266X},
     journal={J. Algebra},
      volume={316},
       pages={560\ndash 577},
         url={https://doi.org/10.1016/j.jalgebra.2006.10.004},
      review={\MR{2356844}},
}

\bib{vasconcelos1994arithmetic}{book}{
      author={Vasconcelos, W.},
       title={Arithmetic of blowup algebras},
      series={London Mathematical Society Lecture Note Series},
   publisher={Cambridge University Press},
        date={1994},
      volume={195},
}

\end{biblist}
\end{bibdiv}
\end{document}